\documentclass[]{interact}
\usepackage{latexsym}
\usepackage{anyfontsize} 
\usepackage{graphicx}
\usepackage{tikz}
\usepackage{subfig}
\usepackage{pgfplots}
\usepgfplotslibrary{fillbetween}
\usepackage{hyperref}
\usepackage{url}
\usepackage{breakurl}
\newcommand{\bburl}[1]{\textcolor{blue}{\url{#1}}}

\usepackage{mathtools}

\usepackage{todonotes}

\makeatletter
\newcommand{\monthyear}[1]{%
  \def\@monthyear{\uppercase{#1}}}
\newcommand{\volnumber}[1]{%
  \def\@volnumber{\uppercase{#1}}}
\makeatother

\newcommand{\bx}{\mathbf{x}}
\newcommand{\by}{\mathbf{y}}

\newcommand{\bbC}{\mathbb{C}}

\newcommand{\bbN}{\mathbb{N}}

\newcommand{\bbQ}{\mathbb{Q}}
\newcommand{\bbR}{\mathbb{R}}

\newcommand{\bbT}{\mathbb{T}}

\newcommand{\bbZ}{\mathbb{Z}}

\newcommand{\cO}{\mathcal{O}}

\newcommand{\frs}{\mathfrak{s}}

\newcommand{\ra}{\rightarrow}

\newcommand{\sm}{\setminus}
\newcommand{\sbse}{\subseteq}

\newcommand{\spse}{\supseteq}

\newcommand{\iu}{\mathrm{i}}

\allowdisplaybreaks

\theoremstyle{definition}
\newtheorem{thm}{Theorem}[section]
\newtheorem{rem}[thm]{Remark}

\newtheorem{lem}[thm]{Lemma}

\newtheorem{prop}[thm]{Proposition}

\newtheorem*{thm*}{Theorem}
\newtheorem*{rem*}{Remark}
\newtheorem*{folg*}{Folgerung}
\newtheorem*{examples*}{Beispiele}
\newtheorem*{ex*}{Beispiel}
\newtheorem*{lem*}{Lemma}
\newtheorem*{prop*}{Proposition}
\newtheorem*{defi*}{Definition}
\newtheorem*{exercise*}{Übung}
\newtheorem*{conj*}{Conjecture}
\newtheorem*{q*}{Question}

\usepackage{cleveref}

\numberwithin{table}{section} 
\numberwithin{figure}{section}

\begin{document}

\monthyear{Month Year}
\volnumber{Volume, Number}
\setcounter{page}{1}

\title{Precise Asymptotics and Exact Formulas for Tensor Product Energies of Fibonacci Lattices}

\author{
\name{Melia Haase\textsuperscript{a} and Nicolas Nagel\textsuperscript{b}\thanks{Corresponding author: Nicolas Nagel (email: \texttt{nicolas.nagel@ricam.oeaw.ac.at}, ORCiD: 0009-0004-3362-3543)}}
\affil{\textsuperscript{a}Chemnitz University of Technology, Faculty of Mathematics, Reichenhainer Straße 39, 09126 Chemnitz, Germany \\ \textsuperscript{b}RICAM, Altenberger Straße 69, 4040 Linz, Austria}
}

\maketitle

\begin{abstract}
	We consider the asymptotics of sums of the form
$$
\frac1{F_n^\sigma} \sum_{m = 1}^{F_n-1} \frac{f(m/F_n)}{\left|{\sin(\pi m/F_n)}\right|^\sigma} \frac{f(F_{n-1}m/F_n)}{\left|{\sin(\pi F_{n-1}m/F_n)}\right|^\sigma}
$$
where $(F_n)_{n \in \bbN} = (1, 1, 2, 3, 5, 8, 13, \dots)$ are the Fibonacci numbers. Such sums appear, for example, in the context of discrepancy theory and numerical integration methods reformulated as energy minimization problems.

We show that for parameters $\sigma > 1$ and a large class of functions $f$ the above sum behaves asymptotically like
$$
C n + D + O\left((1-\varepsilon)^{n}\right)
$$
for some constants $C$ and $D$. These constants can be given via infinite series connected to the Dedekind zeta function over the algebraic number field $\bbQ(\sqrt5)$.

In special cases we even observe simple closed-form expressions for such sums as above, explicitly proving that
$$
\sum_{m=1}^{F_n-1} \frac1{\sin(\pi m/F_n)^2} \frac1{\sin(\pi F_{n-1} m/F_n)^2} = \frac{4n}{75} F_{2n} - \frac{17}{225}F_n^2 - (-1)^n \frac2{15} - \frac19.
$$
\end{abstract}


\section{Introduction}

Let $\bbT = \bbR / \bbZ$ be the one-dimensional unit torus. Throughout, we will use the identification $\bbT \simeq [0, 1)$. We consider the problem of uniformly distributing a given number of points in the two-dimensional torus $[0, 1)^2$. To analyze this problem rigorously we need to quantify the \emph{uniformity} of a finite point set $X \subseteq [0, 1)^2$.

\begin{figure}[t]
						\hspace{1.1cm}\begin{tikzpicture}
							\begin{axis}[
								plot box ratio = 1 1,
								xmin=-1/21,xmax=22/21,ymin=-1/21,ymax=22/21,
								font=\footnotesize,
								height=0.33\textwidth,
								width=0.33\textwidth,
								xticklabel = \empty,
								yticklabel = \empty,
								ytick style={draw=none},
								xtick style={draw=none},
								]
								\addplot[only marks,black,mark=*,mark size=1.5pt,mark options={solid}] coordinates {
									(0.0, 0.625676917054483) (0.022222222222222223, 0.9333333333333333) (0.044444444444444446, 0.03056318486849119) (0.06666666666666667, 0.4119082359417273) (0.08888888888888889, 0.06666666666666667) (0.1111111111111111, 0.3333333333333333) (0.13333333333333333, 0.6) (0.15555555555555556, 0.24994082668700834) (0.17777777777777778, 0.6666666666666666) (0.2, 0.4) (0.2222222222222222, 0.3333333333333333) (0.24444444444444444, 0.6) (0.26666666666666666, 0.7183238984171398) (0.28888888888888886, 0.6666666666666666) (0.3111111111111111, 0.7333333333333333) (0.3333333333333333, 0.3145283383382188) (0.35555555555555557, 0.6) (0.37777777777777777, 0.6666666666666666) (0.4, 0.8902422188194887) (0.4222222222222222, 0.7333333333333333) (0.4444444444444444, 0.629526454390172) (0.4666666666666667, 0.6) (0.4888888888888889, 0.6666666666666666) (0.5111111111111111, 0.3333333333333333) (0.5333333333333333, 0.4) (0.5555555555555556, 0.08978156614769994) (0.5777777777777777, 0.26666666666666666) (0.6, 0.9929096590797771) (0.6222222222222222, 0.3333333333333333) (0.6444444444444445, 0.4) (0.6666666666666666, 0.38581441742653366) (0.6888888888888889, 0.26666666666666666) (0.7111111111111111, 0.3333333333333333) (0.7333333333333333, 0.8012693286377376) (0.7555555555555555, 0.4) (0.7777777777777778, 0.6666666666666666) (0.8, 0.6) (0.8222222222222222, 0.3333333333333333) (0.8444444444444444, 0.25305347568654823) (0.8666666666666667, 0.4) (0.8888888888888888, 0.6666666666666666) (0.9111111111111111, 0.9333333333333333) (0.9333333333333333, 0.7084863370683493) (0.9555555555555556, 0.20872864734946917) (0.9777777777777777, 0.06666666666666667)
								};
								\addplot[gray,mark options={solid}] coordinates {
									(0, -0.05) (0, 1.05)
								};
								\addplot[gray,mark options={solid}] coordinates {
									(1.0, -0.05) (1.0, 1.05)
								};
								\addplot[gray,mark options={solid}] coordinates {
									(-0.05, 0) (1.05, 0)
								};
								\addplot[gray,mark options={solid}] coordinates {
									(-0.05, 1.0) (1.05, 1.0)
								};
							\end{axis}
						\end{tikzpicture} \hspace{5mm}
						\begin{tikzpicture}
							\begin{axis}[
								plot box ratio = 1 1,
								xmin=-1/21,xmax=22/21,ymin=-1/21,ymax=22/21,
								font=\footnotesize,
								height=0.33\textwidth,
								width=0.33\textwidth,
								xticklabel = \empty,
								yticklabel = \empty,
								ytick style={draw=none},
								xtick style={draw=none},
								]
								\addplot[only marks,black,mark=*,mark size=1.5pt,mark options={solid}] coordinates {
									(0.0, 0.06666666666666667) (0.022222222222222223, 0.28888888888888886) (0.044444444444444446, 0.5333333333333333) (0.06666666666666667, 0.13333333333333333) (0.08888888888888889, 0.15555555555555556) (0.1111111111111111, 0.0) (0.13333333333333333, 0.5777777777777777) (0.15555555555555556, 0.4444444444444444) (0.17777777777777778, 0.7555555555555555) (0.2, 0.6222222222222222) (0.2222222222222222, 0.9333333333333333) (0.24444444444444444, 0.2) (0.26666666666666666, 0.7777777777777778) (0.28888888888888886, 0.24444444444444444) (0.3111111111111111, 0.8) (0.3333333333333333, 0.6666666666666666) (0.35555555555555557, 0.6) (0.37777777777777777, 0.9555555555555556) (0.4, 0.8666666666666667) (0.4222222222222222, 0.022222222222222223) (0.4444444444444444, 0.17777777777777778) (0.4666666666666667, 0.8888888888888888) (0.4888888888888889, 0.1111111111111111) (0.5111111111111111, 0.3111111111111111) (0.5333333333333333, 0.2222222222222222) (0.5555555555555556, 0.044444444444444446) (0.5777777777777777, 0.26666666666666666) (0.6, 0.7111111111111111) (0.6222222222222222, 0.9111111111111111) (0.6444444444444445, 0.6444444444444445) (0.6666666666666666, 0.37777777777777777) (0.6888888888888889, 0.5555555555555556) (0.7111111111111111, 0.9777777777777777) (0.7333333333333333, 0.8222222222222222) (0.7555555555555555, 0.5111111111111111) (0.7777777777777778, 0.35555555555555557) (0.8, 0.6888888888888889) (0.8222222222222222, 0.08888888888888889) (0.8444444444444444, 0.4888888888888889) (0.8666666666666667, 0.4222222222222222) (0.8888888888888888, 0.3333333333333333) (0.9111111111111111, 0.7333333333333333) (0.9333333333333333, 0.4666666666666667) (0.9555555555555556, 0.4) (0.9777777777777777, 0.8444444444444444)
								};
								\addplot[gray,mark options={solid}] coordinates {
									(0, -0.05) (0, 1.05)
								};
								\addplot[gray,mark options={solid}] coordinates {
									(1.0, -0.05) (1.0, 1.05)
								};
								\addplot[gray,mark options={solid}] coordinates {
									(-0.05, 0) (1.05, 0)
								};
								\addplot[gray,mark options={solid}] coordinates {
									(-0.05, 1.0) (1.05, 1.0)
								};
							\end{axis}
						\end{tikzpicture} \hspace{5mm}
						\begin{tikzpicture}
							\begin{axis}[
								plot box ratio = 1 1,
								xmin=-1/21,xmax=22/21,ymin=-1/21,ymax=22/21,
								font=\footnotesize,
								height=0.33\textwidth,
								width=0.33\textwidth,
								xticklabel = \empty,
								yticklabel = \empty,
								ytick style={draw=none},
								xtick style={draw=none},
								]
								\addplot[only marks,black,mark=*,mark size=1.5pt,mark options={solid}] coordinates {
									(0.0, 0.8444444444444444) (0.022222222222222223, 0.4888888888888889) (0.044444444444444446, 0.1111111111111111) (0.06666666666666667, 0.7555555555555555) (0.08888888888888889, 0.35555555555555557) (0.1111111111111111, 0.022222222222222223) (0.13333333333333333, 0.6) (0.15555555555555556, 0.2222222222222222) (0.17777777777777778, 0.9111111111111111) (0.2, 0.6444444444444445) (0.2222222222222222, 0.4222222222222222) (0.24444444444444444, 0.06666666666666667) (0.26666666666666666, 0.8) (0.28888888888888886, 0.26666666666666666) (0.3111111111111111, 0.5111111111111111) (0.3333333333333333, 0.9777777777777777) (0.35555555555555557, 0.6888888888888889) (0.37777777777777777, 0.17777777777777778) (0.4, 0.37777777777777777) (0.4222222222222222, 0.8666666666666667) (0.4444444444444444, 0.5777777777777777) (0.4666666666666667, 0.13333333333333333) (0.4888888888888889, 0.7333333333333333) (0.5111111111111111, 0.3333333333333333) (0.5333333333333333, 0.9333333333333333) (0.5555555555555556, 0.4666666666666667) (0.5777777777777777, 0.044444444444444446) (0.6, 0.6222222222222222) (0.6222222222222222, 0.2) (0.6444444444444445, 0.8222222222222222) (0.6666666666666666, 0.28888888888888886) (0.6888888888888889, 0.5333333333333333) (0.7111111111111111, 0.0) (0.7333333333333333, 0.7777777777777778) (0.7555555555555555, 0.4) (0.7777777777777778, 0.08888888888888889) (0.8, 0.6666666666666666) (0.8222222222222222, 0.24444444444444444) (0.8444444444444444, 0.8888888888888888) (0.8666666666666667, 0.4444444444444444) (0.8888888888888888, 0.7111111111111111) (0.9111111111111111, 0.15555555555555556) (0.9333333333333333, 0.5555555555555556) (0.9555555555555556, 0.9555555555555556) (0.9777777777777777, 0.3111111111111111)
								};
								\addplot[gray,mark options={solid}] coordinates {
									(0, -0.05) (0, 1.05)
								};
								\addplot[gray,mark options={solid}] coordinates {
									(1.0, -0.05) (1.0, 1.05)
								};
								\addplot[gray,mark options={solid}] coordinates {
									(-0.05, 0) (1.05, 0)
								};
								\addplot[gray,mark options={solid}] coordinates {
									(-0.05, 1.0) (1.05, 1.0)
								};
							\end{axis}
						\end{tikzpicture}
                        \caption{Three point sets $X_1, X_2, X_3$ (from left to right) of size $N=45$ with $e_{2, 6}(X_1) = 0.037156\dots$, $e_{2, 6}(X_2) = 0.006865\dots$ and $e_{2, 6}(X_3) = 0.001879\dots$ ($(\sigma, p) = (2, 6)$ corresponding to the notion of periodic $L_2$-discrepancy \cite{HKP21, HO16}). Observe how point sets with high and low density regions (\emph{clusters} and \emph{holes} respectively) correspond to bigger $e_{\sigma, p}(X)$.}
\label{fig:uniform}
\end{figure}
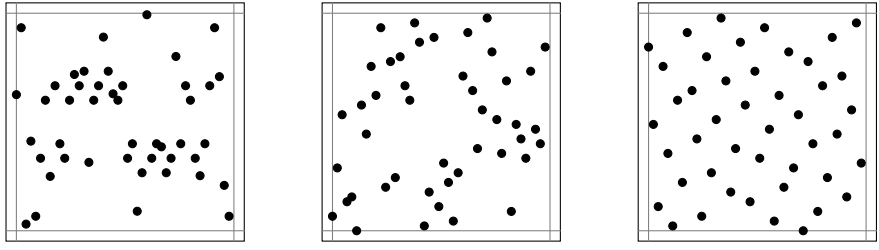

\subsection{Uniformity via energy minimization}

\emph{Discrepancy theory} and the analysis of \emph{quasi-Monte Carlo integration} are natural areas where the notion of the uniformity of a point set shows up. Background and details in these directions can be found for example in \cite{DT97, Mat10, Mon94, Nie78, Nie92, NW10}. For the purposes of this paper, a finite point set $X \sbse [0, 1)^2, \#X = N$ (which may be a multiset, that is it may contain the same point multiple times) will be thought of as uniform if it minimizes
$$
e_{\sigma, p}(X) \coloneqq -1 + \frac1{N^2} \sum_{(x_1, x_2), (y_1, y_2) \in X} K_{\sigma, p}(x_1-y_1) K_{\sigma, p}(x_2-y_2),
$$
where $\sigma > 1$, $p > 0$ are parameters and the function $K_{\sigma, p}: [0, 1) \ra \bbR$ is given by
$$
K_{\sigma, p}(t) \coloneqq 1 + p \sum_{m \neq 0} \frac{\exp(2\pi\iu m t)}{|2\pi m|^\sigma}.
$$
Figure \ref{fig:uniform} illustrates how $e_{\sigma, p}(X)$ corresponds to the ``uniformity'' of $X$. The quantities $e_{\sigma, p}(X)$ can be interpreted as the squared worst-case error of the quasi-Monte Carlo integration algorithm for periodic functions of prescribed smoothness, see for example \cite{BNR25, Dic07, HO16}, formalizing the connection between low $e_{\sigma, p}(X)$ and uniform $X$. The case of $\sigma=2$ corresponds to the notions of \emph{periodic $L_2$-discrepancy} and \emph{diaphony}, see \cite{BNR25, HKP21} for details.

For a given $N \in \bbN$ the main objective is to find a point set $X \sbse [0, 1)^2, \#X=N$ such that for all $Y \sbse [0, 1)^2, \#Y=N$ it holds $e_{\sigma, p}(Y) \geq e_{\sigma, p}(X)$. This is equivalent to minimizing the double sum
$$
E_{\sigma, p}(X) \coloneqq \sum_{(x_1, x_2), (y_1, y_2) \in X} K_{\sigma, p}(x_1-y_1) K_{\sigma, p}(x_2-y_2),
$$
which (in analogy to the \emph{Thomson problem} \cite{Tho04}, also see \cite{BHS19, CK07, HS04}) we interpret as the \emph{energy} of a point configuration $X$ with \emph{potential interaction} $K_{\sigma, p}(x_1-y_1) K_{\sigma, p}(x_2-y_2)$ between two points $(x_1, x_2), (y_1, y_2) \in [0, 1)^2$. If $\sigma = 2s, s \in \bbN$ is an even integer, we have more explicitly
$$
K_{2s, p}(t) = 1 + p \frac{(-1)^{s-1}}{(2s)!} B_{2s}(t)
$$
with the (periodized) \emph{Bernoulli polynomials} $B_{2s}(t)$ of degree $2s$, in general given via
$$
\sum_{n=0}^\infty B_n(t) \frac{x^n}{n!} = \frac{x}{e^x-1} e^{tx}
$$
for $t \in [0, 1)$. The next section introduces a class of point sets which will yield candidates for minimizers of $e_{\sigma, p}(X)$.

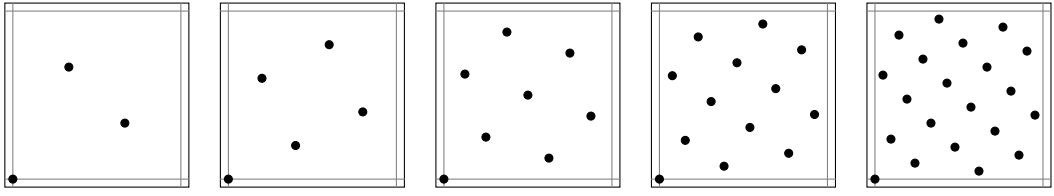
\begin{figure}[t]
						\begin{tikzpicture}
							\begin{axis}[
								plot box ratio = 1 1,
								xmin=-1/21,xmax=22/21,ymin=-1/21,ymax=22/21,
								font=\footnotesize,
								height=0.28\textwidth,
								width=0.28\textwidth,
								xticklabel = \empty,
								yticklabel = \empty,
								ytick style={draw=none},
								xtick style={draw=none},
								]
								\addplot[only marks,black,mark=*,mark size=1.5pt,mark options={solid}] coordinates {
									(0, 0) (1/3, 2/3) (2/3, 1/3)
								};
								\addplot[gray,mark options={solid}] coordinates {
									(0, -0.05) (0, 1.05)
								};
								\addplot[gray,mark options={solid}] coordinates {
									(1.0, -0.05) (1.0, 1.05)
								};
								\addplot[gray,mark options={solid}] coordinates {
									(-0.05, 0) (1.05, 0)
								};
								\addplot[gray,mark options={solid}] coordinates {
									(-0.05, 1.0) (1.05, 1.0)
								};
							\end{axis}
						\end{tikzpicture}
						\begin{tikzpicture}
							\begin{axis}[
								plot box ratio = 1 1,
								xmin=-1/21,xmax=22/21,ymin=-1/21,ymax=22/21,
								font=\footnotesize,
								height=0.28\textwidth,
								width=0.28\textwidth,
								xticklabel = \empty,
								yticklabel = \empty,
								ytick style={draw=none},
								xtick style={draw=none},
								]
								\addplot[only marks,black,mark=*,mark size=1.5pt,mark options={solid}] coordinates {
									(0, 0) (1/5, 3/5) (2/5, 1/5) (3/5, 4/5) (4/5, 2/5)
								};
								\addplot[gray,mark options={solid}] coordinates {
									(0, -0.05) (0, 1.05)
								};
								\addplot[gray,mark options={solid}] coordinates {
									(1.0, -0.05) (1.0, 1.05)
								};
								\addplot[gray,mark options={solid}] coordinates {
									(-0.05, 0) (1.05, 0)
								};
								\addplot[gray,mark options={solid}] coordinates {
									(-0.05, 1.0) (1.05, 1.0)
								};
							\end{axis}
						\end{tikzpicture}
						\begin{tikzpicture}
							\begin{axis}[
								plot box ratio = 1 1,
								xmin=-1/21,xmax=22/21,ymin=-1/21,ymax=22/21,
								font=\footnotesize,
								height=0.28\textwidth,
								width=0.28\textwidth,
								xticklabel = \empty,
								yticklabel = \empty,
								ytick style={draw=none},
								xtick style={draw=none},
								]
								\addplot[only marks,black,mark=*,mark size=1.5pt,mark options={solid}] coordinates {
									(0, 0) (1/8, 5/8) (2/8, 2/8) (3/8, 7/8) (4/8, 4/8) (5/8, 1/8) (6/8, 6/8) (7/8, 3/8)
								};
								\addplot[gray,mark options={solid}] coordinates {
									(0, -0.05) (0, 1.05)
								};
								\addplot[gray,mark options={solid}] coordinates {
									(1.0, -0.05) (1.0, 1.05)
								};
								\addplot[gray,mark options={solid}] coordinates {
									(-0.05, 0) (1.05, 0)
								};
								\addplot[gray,mark options={solid}] coordinates {
									(-0.05, 1.0) (1.05, 1.0)
								};
							\end{axis}
						\end{tikzpicture}
						\begin{tikzpicture}
							\begin{axis}[
								plot box ratio = 1 1,
								xmin=-1/21,xmax=22/21,ymin=-1/21,ymax=22/21,
								font=\footnotesize,
								height=0.28\textwidth,
								width=0.28\textwidth,
								xticklabel = \empty,
								yticklabel = \empty,
								ytick style={draw=none},
								xtick style={draw=none},
								]
								\addplot[only marks,black,mark=*,mark size=1.5pt,mark options={solid}] coordinates {
									(0/13, 0/13) (1/13, 8/13) (2/13, 3/13) (3/13, 11/13) (4/13, 6/13) (5/13, 1/13) (6/13, 9/13) (7/13, 4/13) (8/13, 12/13) (9/13, 7/13) (10/13, 2/13) (11/13, 10/13) (12/13, 5/13)
								};
								\addplot[gray,mark options={solid}] coordinates {
									(0, -0.05) (0, 1.05)
								};
								\addplot[gray,mark options={solid}] coordinates {
									(1.0, -0.05) (1.0, 1.05)
								};
								\addplot[gray,mark options={solid}] coordinates {
									(-0.05, 0) (1.05, 0)
								};
								\addplot[gray,mark options={solid}] coordinates {
									(-0.05, 1.0) (1.05, 1.0)
								};
							\end{axis}
						\end{tikzpicture}
						\begin{tikzpicture}
							\begin{axis}[
								plot box ratio = 1 1,
								xmin=-1/21,xmax=22/21,ymin=-1/21,ymax=22/21,
								font=\footnotesize,
								height=0.28\textwidth,
								width=0.28\textwidth,
								xticklabel = \empty,
								yticklabel = \empty,
								ytick style={draw=none},
								xtick style={draw=none},
								]
								\addplot[only marks,black,mark=*,mark size=1.5pt,mark options={solid}] coordinates {
									(0/21, 0/21) (1/21, 13/21) (2/21, 5/21) (3/21, 18/21) (4/21, 10/21) (5/21, 2/21) (6/21, 15/21) (7/21, 7/21) (8/21, 20/21) (9/21, 12/21) (10/21, 4/21) (11/21, 17/21) (12/21, 9/21) (13/21, 1/21) (14/21, 14/21) (15/21, 6/21) (16/21, 19/21) (17/21, 11/21) (18/21, 3/21) (19/21, 16/21) (20/21, 8/21) 
								};
								\addplot[gray,mark options={solid}] coordinates {
									(0, -0.05) (0, 1.05)
								};
								\addplot[gray,mark options={solid}] coordinates {
									(1.0, -0.05) (1.0, 1.05)
								};
								\addplot[gray,mark options={solid}] coordinates {
									(-0.05, 0) (1.05, 0)
								};
								\addplot[gray,mark options={solid}] coordinates {
									(-0.05, 1.0) (1.05, 1.0)
								};
							\end{axis}
						\end{tikzpicture}
                        \caption{Fibonacci lattices of size $3, 5, 8, 13$ and $21$.}
\label{fig:fib_lats}
\end{figure}

\subsection{Lattices}

Given two numbers $N \in \bbN, h \in \{1, \dots, N\}$ with $\gcd(N, h) = 1$, the \emph{(rational/integration) lattice} $\Lambda_{N, h} \sbse [0, 1)^2$ is given by
$$
\Lambda_{N, h} \coloneqq \left\{\left(\frac kN, \left\{ \frac{hk}N\right\}\right): k=0, 1, \dots, N-1\right\},
$$
where $\{x\} = x-\lfloor x \rfloor$ denotes the fractional part of $x \in \bbR$. Letting $(F_n)_n$ denote the \emph{Fibonacci numbers} given by $F_0=0, F_1=1$ and the recurrence $F_n = F_{n-1} + F_{n-2}$, the special lattices given by
$$
\Phi_n \coloneqq \Lambda_{F_n, F_{n-1}}
$$
are called \emph{Fibonacci lattices}, see Figure \ref{fig:fib_lats}. It is known that Fibonacci lattices achieve the optimal asymptotic rate \cite{DTU18}
$$
e_{\sigma, p}(\Phi_n) \asymp \frac{n}{F_n^\sigma}.
$$
It has been observed on multiple occasions, based on numerical computations \cite{HO16} and partial results \cite{BNR25, Nag26}, also see \cite{NagDiss}, that one even seems to have $e_{\sigma, p}(Y) \geq e_{\sigma, p}(\Phi_n)$ for all $Y \subseteq [0, 1)^2, \# Y = F_n$ given that $0 < p \leq \overline p(\sigma)$ for some constant $\overline p(\sigma) > 0$. That is, Fibonacci lattices are conjectured to be global minimizers for a certain range of parameters. This, however, remains an open problem. In this paper we aim to study the special structure of the Fibonacci lattice by considering the asymptotics of $e_{\sigma, p}(\Phi_n)$ in more detail.

To start, we will consider more general energies of the form
\begin{align} \label{eq:tpe_general}
    E_c(X) \coloneqq \sum_{\bx, \by \in X} c(x_1-y_1) c(x_2-y_2),
\end{align}
where $X \sbse [0, 1)^2$ and $c: [0, 1) \ra \bbR$ is a $1$-periodic, continuous and symmetric function in the sense of $c(t) = c(1-t)$ for $t \in [0, 1)$. Functions $c$ fulfilling these criteria will be called \emph{potentials}. The systematic study of these \emph{tensor product energies} has recently been initiated in \cite{BNR25, Nag25, Nag26}. Such energies behave well together with the structure of lattices as shown by the following result.

\begin{prop} \label{prop:dft_sum}
    Let $\Lambda_{N, h} \sbse [0, 1)^2$ be a lattice and let $c: [0, 1) \ra \bbR$ be a potential. Assume that we have the discrete Fourier expansion
    \begin{align} \label{eq:dft_general}
        c\left(\frac kN\right) = \sum_{m=0}^{N-1} \hat c_N(m) \exp\left(2\pi\iu m\frac{k}{N}\right),
    \end{align}
    valid for all $k = 0, 1, \dots, N-1$. Then
    $$
    E_c(\Lambda_{N, h}) = N^2 \sum_{m=0}^{N-1} \hat c_N(m) \hat c_N(hm),
    $$
    where $\hat c_N$ is regarded as an $N$-periodic function.
\end{prop}
\begin{proof}
    Inserting \eqref{eq:dft_general} into \eqref{eq:tpe_general} one obtains
    \begin{align*}
        E_c(\Lambda_{N, h}) & = \sum_{k, \ell = 0}^{N-1} \left(\sum_{m=0}^{N-1} \hat c_N(m) \exp\left(2\pi\iu m\frac{k-\ell}{N}\right)\right) \\
        & ~~~~~~~~~~ \times \left(\sum_{n=0}^{N-1} \hat c_N(n) \exp\left(2\pi\iu nh\frac{k-\ell}{N}\right)\right) \\
        & = \sum_{m, n = 0}^{N-1} \hat c_N(m) \hat c_N(n) \sum_{k, \ell = 0}^{N-1} \exp\left(2\pi\iu (m+nh)\frac{k-\ell}{N}\right).
    \end{align*}
    The last double sum is zero unless $m+hn \equiv 0 \mod N$, in which case it equates to $N^2$, and thus
    $$
    E_c(\Lambda_{N, h}) = N^2 \sum_{n = 0}^{N-1} \hat c_N(-hn) \hat c_N(n) = N^2 \sum_{m=0}^{N-1} \hat c_N(m) \hat c_N(hm),
    $$
    since, by symmetry $c(t) = c(1-t)$, it holds $\hat c_N(m) = \hat c_N(-m)$.
\end{proof}

If the potential is given by $c(t) = K_{\sigma, p}(t)$, so that $E_c(X) = E_{\sigma, p}(X)$, the discrete Fourier representation takes the form \cite{CK95}
\begin{align} \label{eq:dft_sigma}
    \begin{split}
        & 1 + p \sum_{m \neq 0} \frac{\exp\left(2\pi \iu m \frac kN\right)}{|2\pi m|^{\sigma}} \\
    = ~ & 1 + p \frac{2 \zeta(\sigma)}{(2\pi N)^{\sigma}} + p \sum_{m=1}^{N-1} \frac{\zeta\left(\sigma, \frac mN\right) + \zeta\left(\sigma, 1-\frac mN\right)}{(2 \pi N)^{\sigma}} \exp\left(2\pi\iu m \frac kN\right)
    \end{split}
\end{align}
for $k=0, 1, \dots, N-1$. Here $\zeta(\sigma) = \zeta(\sigma, 1)$ denotes the \emph{Riemann zeta function} and more generally $\zeta(\sigma, a)$ the \emph{Hurwitz zeta function}
$$
\zeta(\sigma, a) = \sum_{n=0}^\infty \frac1{(n+a)^\sigma}.
$$
In the even integer case $\sigma=2s, s \in \bbN$ these expressions can be given more explicitly by \cite{Cvi09}
\begin{align} \label{eq:Bernoulli_dft}
    \begin{split}
        & 1 + p \sum_{m \neq 0} \frac{\exp\left(2\pi\iu m \frac kN\right)}{(2 \pi m)^{2s}} = 1 + p \frac{(-1)^{s-1}}{(2s)!} B_{2s}\left(\frac kN\right) \\
    = ~ & 1 + p \frac{(-1)^{s-1}B_{2s}}{(2s)!N^{2s}} + \frac{p}{(2s-1)! (2N)^{2s}}\sum_{m=1}^{N-1} \frac{g^{(2s-1)}\left(\frac mN\right)}{\pi^{2s-1}} \exp\left(2\pi\iu m \frac kN\right)
    \end{split}
\end{align}
with $g^{(2s-1)}(x)$ denoting the $(2s-1)$-th derivative of $g(x) \coloneqq - \cot(\pi x)$ and $B_{2s} = B_{2s}(0)$ being the \emph{Bernoulli numbers}. In terms of $e_{\sigma, p}(\Lambda_{N,h})$, equivalently the energy $E_{\sigma, p}(\Lambda_{N, h})$, by Proposition \ref{prop:dft_sum} we thus arrive at
\begin{align} \label{eq:wce_expression}
    \begin{split}
        e_{\sigma, p}(\Lambda_{N, h})
    = ~ & p \frac{4 \zeta(\sigma)}{(2\pi N)^{\sigma}} + p^2 \frac{4 \zeta(\sigma)^2}{(2\pi N)^{2\sigma}} \\
    & + \frac{p^2}{(2\pi N)^{2\sigma}} \sum_{m=1}^{N-1} \left(\zeta\left(\sigma, \frac mN\right) + \zeta\left(\sigma, 1-\frac mN\right)\right) \\
    & ~~~~~~~~~~~~~~~~~~~~~ \times \left(\zeta\left(\sigma, \left\{h\frac mN\right\}\right) + \zeta\left(\sigma, 1-\left\{h\frac mN\right\}\right)\right).
    \end{split}
\end{align}
Analyzing the last sum in the case of the Fibonacci lattice $\Phi_n = \Lambda_{F_n, F_{n-1}}$ will be the main objective of this paper.

\subsection{Main result and intuition}

Since $\zeta(\sigma, a) = a^{-\sigma} + \zeta(\sigma, a+1) = a^{-\sigma} + O(1)$ as $a \searrow 0$, we see that for $\sigma > 1$ there is a $1$-periodic and continuously differentiable function $f_\sigma: [0, 1) \ra \bbR$ such that
$$
\zeta(\sigma, a) + \zeta(\sigma, 1-a) = \frac{f_\sigma(a)}{\sin(\pi a)^\sigma}
$$
for all $0 < a < 1$ (with $f_\sigma(0) = f_\sigma(1) = \pi^\sigma$ and $f_\sigma'(0) = f_\sigma'(1) = 0$). 

For the statement below we will say that a $1$-periodic function $f: [0, 1) \rightarrow \bbR$ is \emph{$\alpha$-Hölder continuous} for some $0 < \alpha \leq 1$ if there is a constant $L_f \geq 0$ such that $|f(x)-f(y)| \leq L_f |x-y|^\alpha$ for all $x, y \in [0, 1)$. By the above, $f_\sigma$ is $1$-Hölder continuous as a continuously differentiable function, that is \emph{Lipschitz continuous}.

\begin{thm} \label{thm:asymp}
    Let $\sigma > 1$ and let $f: [0, 1) \ra \bbR$ be $1$-periodic, symmetric (in the sense of $f(t) = f(1-t)$) and $\alpha$-Hölder continuous for some $0 < \alpha \leq 1$. Then
    \begin{align} \label{eq:asymp_expansion}
        \frac1{F_n^\sigma} \sum_{m = 1}^{F_n-1} \frac{f(m/F_n)}{\left|{\sin(\pi m/F_n)}\right|^\sigma} \frac{f(F_{n-1}m/F_n)}{\left|{\sin(\pi F_{n-1}m/F_n)}\right|^\sigma} = Cn + D + O\left(\frac{n^2}{\phi^{\min\{\alpha, \sigma-1\} n/2}}\right)
    \end{align}
    as $n \ra \infty$ for some constants $C=C(\sigma, f), D=D(\sigma, f) \in \bbR$.
\end{thm}

The constants $C$ and $D$ will depend only on $\sigma$ and $f$ and can be determined to arbitrary precision via certain infinite series, see \eqref{C} and \eqref{D} below, where $C$ can be given explicitly via values of a certain Dedekind zeta function, see \eqref{eq:C_exact} below. While the value for $C$ was essentially already determined in \cite{Bor25}, to the knowledge of the authors no investigation into lower-order terms has been carried out before. Our main takeaway here is the following:
\begin{center}
    \textit{If Fibonacci lattices turn out to be minimizers for the quantities $e_{\sigma, p}$ (or even general energies $E_c$), then Theorem \ref{thm:asymp} gives a precise control of their asymptotics via \eqref{eq:wce_expression}.}
\end{center}

Since we want to give a for the most part self-contained account (only using certain number theoretic formulas in the last two parts of this paper for calculating specific values) and do not assume deeper prerequisites in number theory, in particular avoiding class field theory as is frequently employed in this context \cite{Bec14, Bor17, Bor25}, we aim to give some intuition for Theorem \ref{thm:asymp} demonstrated on the example of
\begin{align} \label{sum_demo}
    \Sigma \coloneqq \sum_{m=1}^{F_n-1} \frac1{F_n^2 \sin(\pi m/{F_n})^2 \sin(\pi F_{n-1} m/{F_n})^2}.
\end{align}
This is, up to a simple factor, the final sum in \eqref{eq:wce_expression} with $\sigma = 2s = 2$. Consider the plot in Figure \ref{fig:sum_terms}, which shows the terms of this sum in the case of $n=15$. We see by symmetry that the $m$-th term is equal to the $(F_n-m)$-th one and we can thus approximate
\begin{align} \label{approx_half}
    \Sigma \approx 2 \sum_{1 \leq m < \frac{F_n}2} \frac1{F_n^2 \sin(\pi m/{F_n})^2 \sin(\pi F_{n-1} m/{F_n})^2}.
\end{align}
This possibly removes the term corresponding to $m = F_n/2$, but this will be absorbed into the error term of \eqref{eq:asymp_expansion}. Next, observe that we get a clustering behavior of certain terms around certain values, some of them indicated as dashed lines in Figure \ref{fig:sum_terms}. The terms which cluster around the top-most dashed line correspond to the indices
$$
m \in \{1, 2, 3, 5, 8, 13, 21, 34, \dots\},
$$
the Fibonacci numbers. The terms clustering around the second, third, forth and fifth dashed line from the top respectively correspond to the indices
\begin{align} \label{levels_partition}
    \begin{split}
        m & \in \{6, 10, 16, 26, 42, 68, \dots\}, \\
    m & \in \{4, 7, 11, 18, 29, 47, \dots\}, \\
    m & \in \{9, 15, 24, 39, 63, 102, \dots\}, \\
    m & \in \{14, 19, 23, 31, 37, 50, 60, 81, 97, 131, 157, 212, \dots\}.
    \end{split}
\end{align}
Observe that these sets are either sequences that fulfill the Fibonacci recurrence $U_k = U_{k-1}+U_{k-2}$ or, in the last case, the union of such sequences (namely $\{14, 23, 37, 60, \dots\} \cup \{19, 31, 50, 81, \dots\}$). The strategy is thus to identify these Fibonacci-like sequences to partition the index set $\{m \in \bbN: m < F_n/2\}$ of \eqref{approx_half} into classes of size $\approx n$. In each class, the terms will cluster around certain values $y_1, y_2, \dots$ represented by the dashed lines in Figure \ref{fig:sum_terms}. This is where the main term $Cn$ in \eqref{eq:asymp_expansion} will come from and where we get an expression of the form $C=\sum_i y_i$. For the lower-order term $D+O(n^2 \phi^{-\min\{\alpha, \sigma-1\}n/2})$ we will need to take track of the error that we make in approximating each term from such a class by the corresponding constant $y_i$. This will be the main technicality of the proof.

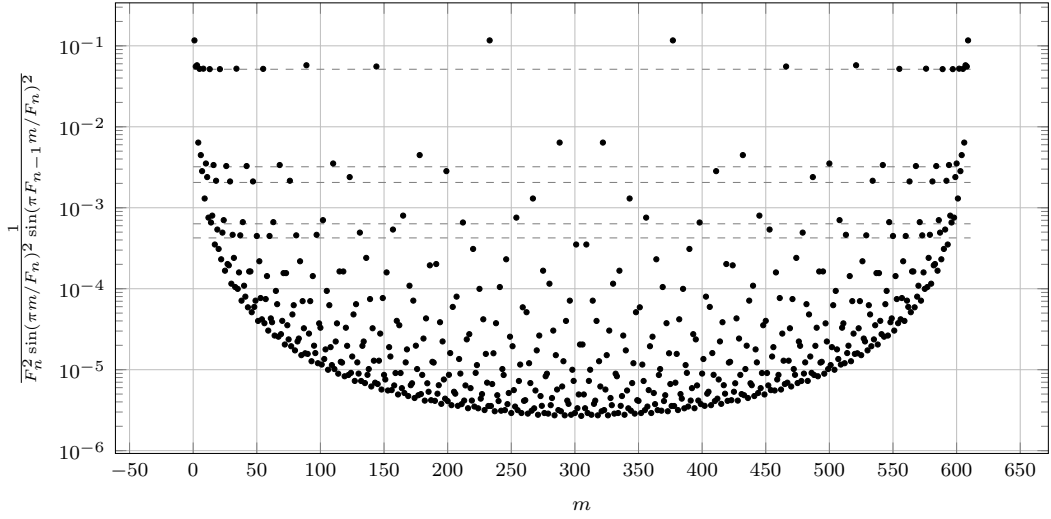
\begin{figure}[t]
		\begin{tikzpicture}[baseline, scale=0.967]
		\begin{axis}[
			font=\footnotesize,
			enlarge x limits=true,
			enlarge y limits=true,
			height=0.54\textwidth,
			grid=major,
			width=\textwidth,
            ymode=log,
			xlabel={$m$},
			ylabel={$ \frac1{F_{n}^2\sin(\pi m/{F_{n}})^2 \sin(\pi F_{n-1} m/{F_{n}})^2}$},
		]
\addplot[black,only marks,mark=*,mark size=1pt,mark options={solid}] coordinates {
(1, 0.1166381915298848) (2, 0.055516816688968226) (3, 0.05750462493187023) (4, 0.006382267977964846) (5, 0.05189640335991287) (6, 0.004470095099502276) (7, 0.0028321333291675605) (8, 0.05219289991631015) (9, 0.001300970234769216) (10, 0.0035206855093075787) (11, 0.0023923146488011564) (12, 0.0007554458093333213) (13, 0.05153210598967) (14, 0.0006574864550795501) (15, 0.0007997410780821538) (16, 0.0033698747249454185) (17, 0.00035163873162537084) (18, 0.0021548904693958555) (19, 0.0005381941632061441) (20, 0.00030988874390652674) (21, 0.05153210598966985) (22, 0.000230911380516451) (23, 0.0004931939918210303) (24, 0.0007030286079447022) (25, 0.0001669668488931378) (26, 0.0032733771378384334) (27, 0.00020181916810961103) (28, 0.00019448598773900202) (29, 0.002112509848338691) (30, 0.00011538139623344728) (31, 0.0004633616486501671) (32, 0.0002403459861350274) (33, 0.00010493508243266495) (34, 0.05219289991631026) (35, 9.958216622791855e-05) (36, 0.00015898265063136622) (37, 0.0004563416145290383) (38, 7.122265395456422e-05) (39, 0.0006643421628518661) (40, 0.00010892946532139962) (41, 7.984921346212387e-05) (42, 0.0032733771378384313) (43, 5.89496060141819e-05) (44, 0.00016329620833960946) (45, 0.0001634620711499372) (46, 5.126255747304683e-05) (47, 0.002112509848338694) (48, 5.960825578249748e-05) (49, 7.127522315729581e-05) (50, 0.0004480285899131873) (51, 4.001566090109401e-05) (52, 0.0002183568842200886) (53, 7.67244424572706e-05) (54, 4.190878439749013e-05) (55, 0.051896403359913114) (56, 3.7511905105994555e-05) (57, 7.45229930106138e-05) (58, 0.00014325215364601685) (59, 3.0408225388455974e-05) (60, 0.00044802858991318785) (61, 4.2999193765116913e-05) (62, 3.855982504818701e-05) (63, 0.000664342162851866) (64, 2.6407701971629794e-05) (65, 9.394909399418996e-05) (66, 6.444526375046289e-05) (67, 2.5555995754046433e-05) (68, 0.0033698747249454155) (69, 2.7466168777001494e-05) (70, 4.0100029426950635e-05) (71, 0.00015591944394837987) (72, 2.0466524100896935e-05) (73, 0.00015591944394837998) (74, 3.5392695658717015e-05) (75, 2.3666641332799526e-05) (76, 0.0021548904693958594) (77, 1.9482179048994016e-05) (78, 4.813456866141823e-05) (79, 6.25966457362305e-05) (80, 1.7315037427182724e-05) (81, 0.00045634161452903907) (82, 2.2323546942920678e-05) (83, 2.42557536429208e-05) (84, 0.00021835688422008786) (85, 1.511471222555366e-05) (86, 7.027823783912405e-05) (87, 3.247657732164904e-05) (88, 1.594828925300342e-05) (89, 0.05750462493187024) (90, 1.5668132877237228e-05) (91, 2.7925086004388582e-05) (92, 7.027823783912401e-05) (93, 1.271925924371732e-05) (94, 0.00014325215364601701) (95, 1.9788949821092848e-05) (96, 1.6041595155675916e-05) (97, 0.0004633616486501651) (98, 1.20213397089835e-05) (99, 3.7341643864017116e-05) (100, 3.299548226166961e-05) (101, 1.1571814161403556e-05) (102, 0.0007030286079447032) (103, 1.3550085498378518e-05) (104, 1.7804260229864835e-05) (105, 9.394909399418995e-05) (106, 9.982364597296383e-06) (107, 6.259664573623055e-05) (108, 1.898813372682245e-05) (109, 1.1398034753334777e-05) (110, 0.003520685509307552) (111, 1.0184595530194175e-05) (112, 2.22946991250961e-05) (113, 3.73416438640171e-05) (114, 8.929002917054287e-06) (115, 0.0001634620711499372) (116, 1.2512666412528735e-05) (117, 1.2239133256692484e-05) (118, 0.000163296208339609) (119, 8.293562404690776e-06) (120, 3.299548226166962e-05) (121, 1.9713629190169837e-05) (122, 8.599838753023382e-06) (123, 0.0023923146488011577) (124, 9.133106633953356e-06) (125, 1.454511083952518e-05) (126, 4.813456866141825e-05) (127, 7.2668612214057666e-06) (128, 6.444526375046293e-05) (129, 1.2307537486441142e-05) (130, 8.955670916177635e-06) (131, 0.0004931939918210289) (132, 7.252280562098308e-06) (133, 1.9713629190169843e-05) (134, 2.2294699125096062e-05) (135, 6.834997740830872e-06) (136, 0.0002403459861350275) (137, 8.641948248639075e-06) (138, 1.0187130640639817e-05) (139, 7.45229930106139e-05) (140, 6.20364159038743e-06) (141, 3.247657732164907e-05) (142, 1.2914576874268852e-05) (143, 6.913311931102109e-06) (144, 0.05551681668896771) (145, 6.6521543900695325e-06) (146, 1.2914576874268852e-05) (147, 2.792508600438855e-05) (148, 5.690356579889444e-06) (149, 7.672444245727062e-05) (150, 8.623698462482308e-06) (151, 7.569240965484014e-06) (152, 0.00015898265063136603) (153, 5.535472153835606e-06) (154, 1.8988133726822444e-05) (155, 1.454511083952518e-05) (156, 5.594483528404123e-06) (157, 0.0005381941632061444) (158, 6.390769356244541e-06) (159, 9.094076370029278e-06) (160, 4.0100029426950594e-05) (161, 4.943229706107123e-06) (162, 3.5392695658717015e-05) (163, 9.094076370029278e-06) (164, 5.91742638857402e-06) (165, 0.0007997410780821532) (166, 5.1519660047071555e-06) (167, 1.2307537486441143e-05) (168, 1.7804260229864842e-05) (169, 4.7244803719402865e-06) (170, 0.00010892946532139964) (171, 6.431948768458625e-06) (172, 6.795467039500667e-06) (173, 7.127522315729575e-05) (174, 4.469333929921144e-06) (175, 1.9788949821092858e-05) (176, 1.0187130640639817e-05) (177, 4.838869437110746e-06) (178, 0.0044700950995022765) (179, 4.997619963018818e-06) (180, 8.623698462482311e-06) (181, 2.4255753642920776e-05) (182, 4.150375118245071e-06) (183, 4.2999193765116913e-05) (184, 6.795467039500665e-06) (185, 5.341603837795088e-06) (186, 0.00019448598773900207) (187, 4.1995070601139335e-06) (188, 1.2512666412528744e-05) (189, 1.223913325669249e-05) (190, 4.121549890971031e-06) (191, 0.00020181916810961103) (192, 5.0550944530992725e-06) (193, 6.431948768458623e-06) (194, 3.8559825048187004e-05) (195, 3.7847646742767677e-06) (196, 2.2323546942920678e-05) (197, 7.569240965484009e-06) (198, 4.3899785144282735e-06) (199, 0.002832133329167581) (200, 4.098886031249636e-06) (201, 8.641948248639077e-06) (202, 1.604159515567591e-05) (203, 3.6458074556673797e-06) (204, 5.960825578249745e-05) (205, 5.341603837795088e-06) (206, 5.0550944530992725e-06) (207, 7.984921346212355e-05) (208, 3.578287917496178e-06) (209, 1.3550085498378514e-05) (210, 8.955670916177628e-06) (211, 3.7558567455823606e-06) (212, 0.00065748645507955) (213, 4.157753707838223e-06) (214, 6.39076935624454e-06) (215, 2.366664133279951e-05) (216, 3.3425221981662044e-06) (217, 2.7466168777001494e-05) (218, 5.91742638857402e-06) (219, 4.157753707838223e-06) (220, 0.00030988874390652516) (221, 3.506043961942976e-06) (222, 9.133106633953356e-06) (223, 1.1398034753334772e-05) (224, 3.334908169770133e-06) (225, 9.958216622791854e-05) (226, 4.3899785144282735e-06) (227, 4.99761996301882e-06) (228, 4.1908784397490143e-05) (229, 3.17237738849278e-06) (230, 1.5668132877237228e-05) (231, 6.913311931102106e-06) (232, 3.561559606350522e-06) (233, 0.1166381915298848) (234, 3.5615596063505193e-06) (235, 6.6521543900695325e-06) (236, 1.5948289253003416e-05) (237, 3.066785703333689e-06) (238, 3.7511905105994555e-05) (239, 4.838869437110747e-06) (240, 4.098886031249636e-06) (241, 0.00010493508243266459) (242, 3.115484400068336e-06) (243, 1.0184595530194177e-05) (244, 8.599838753023377e-06) (245, 3.1667156228086236e-06) (246, 0.000230911380516451) (247, 3.7558567455823606e-06) (248, 5.151966004707158e-06) (249, 2.5555995754046453e-05) (250, 2.9170930661775636e-06) (251, 1.9482179048994016e-05) (252, 5.59448352840412e-06) (253, 3.506043961942976e-06) (254, 0.0007554458093333183) (255, 3.166715622808623e-06) (256, 7.252280562098308e-06) (257, 1.1571814161403534e-05) (258, 2.9158335601320273e-06) (259, 5.89496060141819e-05) (260, 4.121549890971029e-06) (261, 4.199507060113934e-06) (262, 5.126255747304693e-05) (263, 2.8683434811079863e-06) (264, 1.20213397089835e-05) (265, 6.8349977408308695e-06) (266, 3.115484400068335e-06) (267, 0.0013009702347692163) (268, 3.3349081697701337e-06) (269, 5.535472153835606e-06) (270, 1.7315037427182697e-05) (271, 2.7762760163632173e-06) (272, 2.6407701971629794e-05) (273, 4.724480371940282e-06) (274, 3.578287917496178e-06) (275, 0.0001669668488931375) (276, 2.9158335601320273e-06) (277, 8.293562404690777e-06) (278, 8.929002917054289e-06) (279, 2.8683434811079863e-06) (280, 0.00011538139623344728) (281, 3.645807455667381e-06) (282, 4.469333929921144e-06) (283, 3.040822538845594e-05) (284, 2.73135698829063e-06) (285, 1.5114712225553656e-05) (286, 5.690356579889445e-06) (287, 3.17237738849278e-06) (288, 0.0063822679779645865) (289, 3.0667857033336877e-06) (290, 6.203641590387429e-06) (291, 1.2719259243717306e-05) (292, 2.73135698829063e-06) (293, 4.001566090109401e-05) (294, 4.15037511824507e-06) (295, 3.7847646742767677e-06) (296, 7.122265395456417e-05) (297, 2.776276016363217e-06) (298, 9.982364597296383e-06) (299, 7.2668612214057666e-06) (300, 2.9170930661775644e-06) (301, 0.00035163873162537084) (302, 3.3425221981662036e-06) (303, 4.943229706107123e-06) (304, 2.046652410089692e-05) (305, 2.6874496103198067e-06) (306, 2.0466524100896932e-05) (307, 4.943229706107119e-06) (308, 3.342522198166205e-06) (309, 0.0003516387316253713) (310, 2.9170930661775644e-06) (311, 7.2668612214057666e-06) (312, 9.982364597296388e-06) (313, 2.776276016363217e-06) (314, 7.122265395456424e-05) (315, 3.7847646742767677e-06) (316, 4.150375118245071e-06) (317, 4.001566090109399e-05) (318, 2.73135698829063e-06) (319, 1.2719259243717322e-05) (320, 6.203641590387426e-06) (321, 3.066785703333689e-06) (322, 0.006382267977964846) (323, 3.1723773884927797e-06) (324, 5.690356579889443e-06) (325, 1.5114712225553646e-05) (326, 2.73135698829063e-06) (327, 3.0408225388455974e-05) (328, 4.469333929921144e-06) (329, 3.6458074556673797e-06) (330, 0.00011538139623344675) (331, 2.8683434811079863e-06) (332, 8.929002917054285e-06) (333, 8.29356240469077e-06) (334, 2.9158335601320273e-06) (335, 0.0001669668488931378) (336, 3.578287917496178e-06) (337, 4.7244803719402865e-06) (338, 2.6407701971629794e-05) (339, 2.7762760163632173e-06) (340, 1.731503742718272e-05) (341, 5.535472153835607e-06) (342, 3.3349081697701337e-06) (343, 0.0013009702347692028) (344, 3.115484400068335e-06) (345, 6.834997740830871e-06) (346, 1.2021339708983495e-05) (347, 2.8683434811079863e-06) (348, 5.126255747304683e-05) (349, 4.199507060113934e-06) (350, 4.121549890971029e-06) (351, 5.8949606014181695e-05) (352, 2.915833560132027e-06) (353, 1.1571814161403556e-05) (354, 7.252280562098301e-06) (355, 3.1667156228086236e-06) (356, 0.0007554458093333212) (357, 3.506043961942976e-06) (358, 5.594483528404125e-06) (359, 1.948217904899402e-05) (360, 2.9170930661775636e-06) (361, 2.5555995754046426e-05) (362, 5.151966004707157e-06) (363, 3.7558567455823606e-06) (364, 0.00023091138051645012) (365, 3.166715622808623e-06) (366, 8.59983875302338e-06) (367, 1.0184595530194173e-05) (368, 3.115484400068336e-06) (369, 0.00010493508243266495) (370, 4.098886031249633e-06) (371, 4.838869437110746e-06) (372, 3.751190510599458e-05) (373, 3.066785703333688e-06) (374, 1.5948289253003416e-05) (375, 6.652154390069528e-06) (376, 3.561559606350522e-06) (377, 0.11663819152988092) (378, 3.5615596063505193e-06) (379, 6.913311931102109e-06) (380, 1.5668132877237204e-05) (381, 3.1723773884927792e-06) (382, 4.190878439749013e-05) (383, 4.997619963018815e-06) (384, 4.389978514428275e-06) (385, 9.958216622791838e-05) (386, 3.334908169770133e-06) (387, 1.1398034753334777e-05) (388, 9.133106633953353e-06) (389, 3.506043961942976e-06) (390, 0.0003098887439065267) (391, 4.157753707838223e-06) (392, 5.917426388574019e-06) (393, 2.746616877700145e-05) (394, 3.3425221981662036e-06) (395, 2.366664133279953e-05) (396, 6.390769356244538e-06) (397, 4.157753707838223e-06) (398, 0.000657486455079549) (399, 3.7558567455823606e-06) (400, 8.955670916177628e-06) (401, 1.3550085498378504e-05) (402, 3.578287917496178e-06) (403, 7.984921346212386e-05) (404, 5.0550944530992725e-06) (405, 5.341603837795086e-06) (406, 5.960825578249742e-05) (407, 3.645807455667381e-06) (408, 1.6041595155675916e-05) (409, 8.641948248639075e-06) (410, 4.098886031249633e-06) (411, 0.0028321333291675592) (412, 4.389978514428275e-06) (413, 7.569240965484012e-06) (414, 2.2323546942920648e-05) (415, 3.7847646742767677e-06) (416, 3.8559825048187e-05) (417, 6.431948768458622e-06) (418, 5.0550944530992725e-06) (419, 0.00020181916810961092) (420, 4.121549890971028e-06) (421, 1.2239133256692477e-05) (422, 1.2512666412528735e-05) (423, 4.1995070601139335e-06) (424, 0.00019448598773900202) (425, 5.341603837795086e-06) (426, 6.795467039500667e-06) (427, 4.299919376511688e-05) (428, 4.15037511824507e-06) (429, 2.425575364292079e-05) (430, 8.6236984624823e-06) (431, 4.997619963018817e-06) (432, 0.004470095099502304) (433, 4.838869437110747e-06) (434, 1.0187130640639817e-05) (435, 1.9788949821092865e-05) (436, 4.469333929921144e-06) (437, 7.127522315729581e-05) (438, 6.795467039500665e-06) (439, 6.431948768458621e-06) (440, 0.00010892946532139964) (441, 4.724480371940283e-06) (442, 1.7804260229864832e-05) (443, 1.2307537486441132e-05) (444, 5.151966004707159e-06) (445, 0.000799741078082154) (446, 5.917426388574019e-06) (447, 9.09407637002928e-06) (448, 3.5392695658717e-05) (449, 4.943229706107119e-06) (450, 4.010002942695064e-05) (451, 9.094076370029281e-06) (452, 6.390769356244538e-06) (453, 0.0005381941632061455) (454, 5.594483528404122e-06) (455, 1.4545110839525184e-05) (456, 1.8988133726822434e-05) (457, 5.535472153835608e-06) (458, 0.00015898265063136612) (459, 7.569240965484009e-06) (460, 8.623698462482305e-06) (461, 7.672444245727039e-05) (462, 5.690356579889443e-06) (463, 2.7925086004388572e-05) (464, 1.2914576874268842e-05) (465, 6.652154390069528e-06) (466, 0.05551681668896824) (467, 6.913311931102106e-06) (468, 1.291457687426884e-05) (469, 3.247657732164904e-05) (470, 6.203641590387426e-06) (471, 7.452299301061372e-05) (472, 1.0187130640639817e-05) (473, 8.641948248639077e-06) (474, 0.00024034598613502663) (475, 6.834997740830868e-06) (476, 2.2294699125096096e-05) (477, 1.971362919016983e-05) (478, 7.252280562098303e-06) (479, 0.0004931939918210299) (480, 8.95567091617762e-06) (481, 1.2307537486441137e-05) (482, 6.44452637504629e-05) (483, 7.2668612214057666e-06) (484, 4.813456866141824e-05) (485, 1.4545110839525184e-05) (486, 9.133106633953355e-06) (487, 0.0023923146488011426) (488, 8.599838753023377e-06) (489, 1.9713629190169833e-05) (490, 3.2995482261669596e-05) (491, 8.293562404690769e-06) (492, 0.00016329620833960932) (493, 1.2239133256692484e-05) (494, 1.2512666412528742e-05) (495, 0.0001634620711499375) (496, 8.92900291705429e-06) (497, 3.7341643864017116e-05) (498, 2.2294699125096056e-05) (499, 1.0184595530194177e-05) (500, 0.0035206855093075774) (501, 1.1398034753334772e-05) (502, 1.8988133726822427e-05) (503, 6.259664573623039e-05) (504, 9.982364597296388e-06) (505, 9.394909399418999e-05) (506, 1.7804260229864842e-05) (507, 1.3550085498378504e-05) (508, 0.000703028607944701) (509, 1.1571814161403534e-05) (510, 3.299548226166961e-05) (511, 3.734164386401711e-05) (512, 1.2021339708983495e-05) (513, 0.00046336164865016686) (514, 1.6041595155675913e-05) (515, 1.9788949821092875e-05) (516, 0.00014325215364601664) (517, 1.2719259243717306e-05) (518, 7.027823783912396e-05) (519, 2.792508600438854e-05) (520, 1.5668132877237204e-05) (521, 0.05750462493186686) (522, 1.594828925300342e-05) (523, 3.247657732164906e-05) (524, 7.027823783912393e-05) (525, 1.5114712225553646e-05) (526, 0.00021835688422008846) (527, 2.4255753642920766e-05) (528, 2.232354694292065e-05) (529, 0.0004563416145290378) (530, 1.73150374271827e-05) (531, 6.259664573623045e-05) (532, 4.813456866141826e-05) (533, 1.948217904899402e-05) (534, 0.002154890469395855) (535, 2.3666641332799516e-05) (536, 3.5392695658717e-05) (537, 0.00015591944394837963) (538, 2.046652410089692e-05) (539, 0.00015591944394837974) (540, 4.010002942695059e-05) (541, 2.746616877700145e-05) (542, 0.00336987472494542) (543, 2.555599575404645e-05) (544, 6.444526375046293e-05) (545, 9.394909399418999e-05) (546, 2.6407701971629794e-05) (547, 0.0006643421628518643) (548, 3.855982504818699e-05) (549, 4.299919376511688e-05) (550, 0.00044802858991318704) (551, 3.040822538845594e-05) (552, 0.0001432521536460168) (553, 7.452299301061382e-05) (554, 3.751190510599458e-05) (555, 0.05189640335991292) (556, 4.190878439749014e-05) (557, 7.672444245727039e-05) (558, 0.00021835688422008773) (559, 4.0015660901093986e-05) (560, 0.00044802858991318747) (561, 7.127522315729574e-05) (562, 5.9608255782497405e-05) (563, 0.0021125098483386912) (564, 5.126255747304694e-05) (565, 0.00016346207114993746) (566, 0.0001632962083396089) (567, 5.8949606014181695e-05) (568, 0.0032733771378384226) (569, 7.984921346212353e-05) (570, 0.00010892946532139968) (571, 0.0006643421628518642) (572, 7.122265395456417e-05) (573, 0.00045634161452903847) (574, 0.00015898265063136587) (575, 9.958216622791836e-05) (576, 0.05219289991631057) (577, 0.0001049350824326646) (578, 0.0002403459861350267) (579, 0.0004633616486501649) (580, 0.00011538139623344675) (581, 0.0021125098483386943) (582, 0.00019448598773900207) (583, 0.00020181916810961095) (584, 0.0032733771378384204) (585, 0.0001669668488931375) (586, 0.0007030286079447019) (587, 0.0004931939918210285) (588, 0.00023091138051645015) (589, 0.05153210598966978) (590, 0.00030988874390652506) (591, 0.0005381941632061457) (592, 0.0021548904693958586) (593, 0.0003516387316253713) (594, 0.0033698747249454168) (595, 0.0007997410780821533) (596, 0.0006574864550795489) (597, 0.051532105989669626) (598, 0.000755445809333318) (599, 0.002392314648801144) (600, 0.0035206855093075505) (601, 0.001300970234769203) (602, 0.052192899916310695) (603, 0.0028321333291675796) (604, 0.004470095099502306) (605, 0.05189640335991316) (606, 0.006382267977964585) (607, 0.057504624931866845) (608, 0.05551681668896773) (609, 0.11663819152988092) 
};

\addplot [forget plot,gray,domain=0:611, samples=2, dashed]{0.05132991127342169};
\addplot [forget plot,gray,domain=0:611, samples=2, dashed]{0.0032081194545888554};
\addplot [forget plot,gray,domain=0:611, samples=2, dashed]{0.0020531964509368675};
\addplot [forget plot,gray,domain=0:611, samples=2, dashed]{0.000633702608313848};
\addplot [forget plot,gray,domain=0:611, samples=2, dashed]{0.00042421414275555116};

		\end{axis}
		\end{tikzpicture}
\caption{Terms of \eqref{sum_demo} for $n=15, F_n=610, F_{n-1}=377$. The dashed lines are at $y \in \left\{\frac5{\pi^4}, \frac5{16\pi^4}, \frac5{25 \pi^4}, \frac5{81 \pi^4}, \frac5{121 \pi^4}\right\}$.}
\label{fig:sum_terms}
\end{figure}

\section{The Wythoff array} \label{sec:wythoff}

As seen above, we will need to partition the index set $\{m \in \bbN: m < F_n/2\}$ in \eqref{approx_half} into parts on which the terms are controllable. This partition will be introduced in this section.

\subsection{Basic definitions and properties}

Let
$$
\phi = \frac{1+\sqrt5}{2}
$$
denote the golden ration. The \emph{Wythoff array} $(W_{i, k})_{k \in \bbN}, i \in \bbN$ \cite{Mor80} is given by
$$
W_{i, 1} = \lfloor \phi \lfloor \phi i\rfloor\rfloor \quad \text{and} \quad W_{i, 2} = \lfloor \phi^2 \lfloor \phi i\rfloor\rfloor
$$
for $k=1, 2$ and via the recurrence
\begin{align} \label{eq:recurrence}
    W_{i, k} = W_{i, k-1} + W_{i, k-2}
\end{align}
for $k \geq 3$. Up to a shift in the index, the first sequence $(W_{1, k})_k$ is given by the Fibonacci numbers $W_{1, k} = F_{k+1}$. Some basic and classic results will be summarized in the following statement.

\begin{thm} \label{thm:wythoff_basics}
    Let $(W_{i, k})_{k \in \bbN}, i \in \bbN$ be the Wythoff array.
    \begin{itemize}
        \item [(i)] For $i, k \in\bbN$ it holds
        $$
        W_{i, k} = \frac1{\sqrt5} \left(w_+(i) \phi^k - w_-(i) \left(-\frac1\phi\right)^k \right)
        $$
        with $w_+(i) = \phi^{-2} W_{i, 1} + \phi^{-1} W_{i, 2}$ and $w_-(i) = \phi^2 W_{i, 1} - \phi W_{i, 2}$.

        \item [(ii)] For $i \in \bbN$ set $\eta_i \coloneqq -w_+(i) w_-(i) \in \bbN$. Then for all $k \geq 2$ it holds
        $$
        \eta_i = (-1)^k \left(W_{i, k}^2 - W_{i, k} W_{i, k-1} - W_{i, k-1}^2\right).
        $$

        \item[(iii)] For $i, k \in \bbN$ it holds
        $$
        W_{i, k+1} = \lfloor \phi W_{i, k} + \phi^{-1}\rfloor.
        $$

        \item [(iv)] For every $N \in \bbN$ there is a unique tuple $(i, k) \in \bbN^2$ such that $N = W_{i, k}$.

        \item [(v)] If $(U_n)_{n \in \bbN}$ is a sequence with $U_1, U_2 \in \bbN$ fulfilling the Fibonacci recurrence $U_n = U_{n-1}+U_{n-2}$, then there is a unique $i \in \bbN$ for which there are $\nu \in \bbN, \kappa \in \bbZ$ such that $U_n = W_{i, n+\kappa}$ for all $n \geq \nu$.
    \end{itemize}
\end{thm}
\begin{proof}
    Points (i) and (ii) are simple consequences of the recurrence \eqref{eq:recurrence}. For $\eta_i > 0$ we just note that clearly $w_+(i) > 0$ and from Lemma \ref{lem:floor} (i) with $N = \lfloor \phi i \rfloor$ we have
    $$
    w_-(i) = \phi^2 \lfloor \phi \lfloor \phi i\rfloor\rfloor - \phi \lfloor \phi^2 \lfloor \phi i\rfloor\rfloor < 0
    $$
    (also see \eqref{ineq:w_minus} below). We will show (iii) by induction. From Lemma \ref{lem:floor} (ii) with $N = \lfloor \phi i\rfloor$ we obtain the base case $W_{i, 2} = \lfloor \phi W_{i, 1} + \phi^{-1}\rfloor$ . By Lemma \ref{lem:floor} (iii) with $N = W_{i, k-1}$ we then have inductively for $k \geq 2$
    \begin{align*}
        W_{i, k+1} & = W_{i, k} + W_{i, k-1} = \lfloor \phi W_{i, k-1} + \phi^{-1}\rfloor + W_{i, k-1} \\
        & = \lfloor \phi \lfloor \phi W_{i, k-1} + \phi^{-1}\rfloor + \phi^{-1}\rfloor = \lfloor \phi W_{i, k} + \phi^{-1}\rfloor.
    \end{align*}
    Points (iv) and (v) are proven in \cite{Mor80}.
\end{proof}

Point (iv) can be rephrased as stating that the sets $\{W_{i, k}: k \in \bbN\}, i \in \bbN$ partition the positive integers $\bbN$. The number $\eta_i$ can be considered as an invariant of the sequence $(W_{i, k})_k$. The initial portion of the Wythoff array is given in the following table:
\begin{center}
	\begin{tabular}{r||c||c|c|c|c|c|c}
		$i$ & $\eta_i$ & $W_{i, 1}$ & $W_{i, 2}$ & $W_{i, 3}$ & $W_{i, 4}$ & $W_{i, 5}$ & $W_{i, 6}$ \\
		\hline
		$1$ & $1$ & $1$ & $2$ & $3$ & $5$ & $8$ & $13$ \\
		\hline
		$2$ & $5$ & $4$ & $7$ & $11$ & $18$ & $29$ & $47$ \\
		\hline
		$3$ & $4$ & $6$ & $10$ & $16$ & $26$ & $42$ & $68$ \\
		\hline
		$4$ & $9$ & $9$ & $15$ & $24$ & $39$ & $63$ & $102$ \\
		\hline
		$5$ & $16$ & $12$ & $20$ & $32$ & $52$ & $84$ & $136$ \\
		\hline
		$6$ & $11$ & $14$ & $23$ & $37$ & $60$ & $97$ & $157$ \\
		\hline
		$7$ & $19$ & $17$ & $28$ & $45$ & $73$ & $118$ & $191$ \\
		\hline
		$8$ & $11$ & $19$ & $31$ & $50$ & $81$ & $131$ & $212$
	\end{tabular}
\end{center}
Observe that the rows $(W_{i, k})_k$ for those $i$ such that $\eta_i \in \{1, 4, 5, 9, 11\}$ correspond to the observed sequences in \eqref{levels_partition}.

\begin{rem} \label{rem:w_numbers}
    Since for $i \in \bbN$ it holds
    $$
    \phi \lfloor \phi i\rfloor - (\lfloor\phi i \rfloor + i -1 ) = \phi^{-1} \lfloor \phi i\rfloor - i + 1 = 1 - \phi^{-1} \{\phi i\} \in (\phi^{-2}, 1),
    $$
    we have $W_{i, 1} = \lfloor \phi \lfloor \phi i\rfloor\rfloor = \lfloor \phi i \rfloor + i - 1$. Then also
    $$
    W_{i, 2} = \lfloor \phi^2 \lfloor \phi i\rfloor \rfloor = \lfloor \phi \lfloor \phi i\rfloor \rfloor + \lfloor \phi i\rfloor = 2\lfloor \phi i \rfloor + i - 1
    $$
    and in general
    \begin{align} \label{eq:W_fib}
        W_{i, k} = F_{k+1} \lfloor \phi i\rfloor + F_{k}(i-1).
    \end{align}
    By Theorem \ref{thm:wythoff_basics} (i) thus
    \begin{align}\label{ineq:wp_wm_expressions}
        w_+(i) = \phi\lfloor\phi i\rfloor + i - 1 \quad \text{and} \quad w_-(i) = -\phi^{-1} \lfloor\phi i\rfloor + i - 1 = \phi^{-1} \{\phi i\} - 1.
    \end{align}
    In particular we see that
    \begin{align} \label{ineq:w_minus}
        \phi^{-2} < -w_-(i) < 1
    \end{align}
    for $i \in \bbN$. For $\eta_i$ we obtain
    $$
    \eta_i = \sqrt5 i (\phi - \{\phi i\})-\{\phi i\}(1-\{\phi i\}) - 1.
    $$
    The quantities $w_+(i)$ and $\eta_i$ thus grow linearly in $i$ with the explicit bounds
    \begin{align} \label{lin_growth}
        \phi i \leq w_+(i) < (\phi^2+1)i \quad \text{and} \quad i \leq \eta_i < (\phi^2+1) i
    \end{align}
    holding for all $i \in \bbN$.
\end{rem}

The Wythoff array will be an essential tool in analyzing the sum in \eqref{eq:asymp_expansion}. Important aspects will be discussed in the next part.

\subsection{Dual array}

In \eqref{approx_half} we have seen that we are only interested in those terms $W_{i, k}$ which are strictly less than $F_n/2$. In this subsection we will analyze which terms these are. To understand the edge case, we will first look where equality occurs. Notice that $F_n$ for $n \in \bbN$ is even if and only if $n = 3\ell$ for some $\ell \in \bbN$.

\begin{lem} \label{lem:half_fib}
	For $i, k, n \in \bbN$ it holds
	$$
	W_{i, k} = \frac{F_n}2
	$$
	if and only if
	$$
	(i, k, n) = \left(\frac{F_{3\ell-2}+1}2, 1, 3 \ell\right)
	$$
	for some $\ell \in \bbN$.
\end{lem}
\begin{proof}
    Since by Theorem \ref{thm:wythoff_basics} (iv) the equation $W_{i, k} = {F_n}/2$ has a unique solution in $i, k \in \bbN$ for a fixed $n=3\ell$, it suffices to show that
    $$
    W_{\frac12(F_{3\ell-2}+1), 1} = \frac{F_{3\ell}}2
    $$
    holds for all $\ell \in \bbN$. We will use the well-known formula
    \begin{align} \label{eq:binet}
        F_n = \frac{1}{\sqrt5} \left( \phi^n - \left(-\frac1\phi\right)^n \right)
    \end{align}
    and the straightforward facts that
    $$
    \frac{F_{3\ell-2}+1}2, \frac{F_{3\ell-1}+1}2, \frac{F_{3\ell}}2 \in \bbN.
    $$
    By
    $$
    \phi \frac{F_{3\ell-2}+1}2 - \frac{F_{3\ell-1}+1}2 = \frac{\phi-1}2 - \frac12 \left(-\frac1\phi\right)^{3\ell-2} \in (0, 1)
    $$
    and
    $$
    \phi \frac{F_{3\ell-1}+1}2 - \frac{F_{3\ell}}2 = \frac{\phi}2 - \frac12 \left(-\frac1\phi\right)^{3\ell-1} \in (0, 1)
    $$
    we have
    $$
    W_{\frac12(F_{3\ell-2}+1), 1} = \left\lfloor \phi \left\lfloor \phi \frac{F_{3\ell-2}+1}2\right\rfloor \right\rfloor = \left\lfloor \phi \frac{F_{3\ell-1}+1}2 \right\rfloor = \frac{F_{3\ell}}2
    $$
    as desired.
\end{proof}

In particular we see that the numbers $F_{3\ell}/2$ always appear at the very start of a sequence $(W_{i, k})_k$.

\begin{rem}
    Having shown above that
    $$
    \left\lfloor \phi \frac{F_{3\ell-2}+1}2\right\rfloor = \frac{F_{3\ell-1}+1}2 \quad \text{and} \quad \left\lfloor \phi \frac{F_{3\ell-1}+1}2 \right\rfloor = \frac{F_{3\ell}}2,
    $$
    it is natural to ask if one can similarly go from $\frac12 F_{3\ell}$ to $\frac12(F_{3\ell+1}+1)$. This is indeed the case with some modification. By
    $$
	\phi \frac{F_{3\ell}}2 + \phi^{-1} - \frac{F_{3\ell+1}+1}2 = \frac{2-\phi}{2\phi} - \frac12 \left(-\frac1\phi\right)^{3\ell} \in (0, 1)
	$$
	we get
	\begin{align} \label{eq:rem_cont}
		\left\lfloor \phi \frac{F_{3\ell}}2  + \phi^{-1} \right\rfloor = \frac{F_{3\ell+1}+1}2.
	\end{align}
    Note the similarities with Theorem \ref{thm:wythoff_basics} (iii).
\end{rem}

We now characterize the indices $k$ for which $W_{i, k} < F_n/2$ holds.

\begin{thm} \label{thm:const_shift}
	For every $i \in \bbN$ there is a $\mu_i \in \bbN$ such that for all $n, k \in \bbN$ it holds
	$$
	W_{i, k} < \frac12 F_n \quad \text{if and only if} \quad n-k > \mu_i.
	$$
\end{thm}
\begin{proof}
    The proof is via induction on $k$. The case of $k=1$ gives us the value for $\mu_i$. For the induction step, the cases $k=2$ and $k \geq 3$ will be considered separately. To better organize the proof we will do it in multiple small steps.

    \textit{Step 1, $k=1$.} Set
    \begin{align} \label{mu_i}
        \mu_i \coloneqq \min\left\{m \in \bbN: W_{i, 1} < \frac12 F_{m+2}\right\} \geq 2,
    \end{align}
    so that
    \begin{align} \label{ineq:k=1}
        \frac12 F_{\mu_i+1} \leq W_{i, 1} < \frac12 F_{\mu_i+2}.
    \end{align}
    By the strict monotonicity of $(F_n)_{n \geq 2}$ this is equivalent to the statement for $k=1$.

    \textit{Step 2, $k=2$.} Before coming to the case of general $k$ we will first show that
    \begin{align} \label{ineq:k=2}
        \frac12 F_{\mu_i+2} < W_{i, 2} < \frac12 F_{\mu_i+3},
    \end{align}
    which is the statement for $k=2$ in a slightly stronger form.

    \textit{Step 2.1, the right inequality.} Starting from \eqref{ineq:k=1}, since $W_{i, 1}, F_{\mu_i+2} \in \bbN$ it holds
    $$
    W_{i, 1} < \frac12(2W_{i, 1}+1) \leq \frac12 F_{\mu_i+2}.
    $$
    By Lemma \ref{lem:floor} (iv) with $N=W_{i, 1}$ and Theorem \ref{thm:wythoff_basics} (iii) we then have
    \begin{align*}
        W_{i, 2} & = \lfloor \phi W_{i, 1} + \phi^{-1}\rfloor < \frac12 \lfloor \phi (2W_{i, 1}+1) + \phi^{-1}\rfloor \\
        & \leq \frac12 \lfloor \phi F_{\mu_i+2} + \phi^{-1}\rfloor = \frac12 F_{\mu_i+3},
    \end{align*}
    giving the right inequality in \eqref{ineq:k=2}.

    \textit{Step 2.2, the left inequality.} For the left inequality we need to consider two cases.

    \textit{Step 2.2.1, the case of a strict inequality.} If $\frac12 F_{\mu_i+1} < W_{i, 1}$, then the left inequality of \eqref{ineq:k=2} follows in an analogous manner as in step 2.1, now using the left inequality of Lemma \ref{lem:floor} (iv).

    \textit{Step 2.2.2, the case of equality.} If $\frac12 F_{\mu_i+1} = W_{i, 1}$ then by Lemma \ref{lem:half_fib} we have (among others) $F_{\mu_i+1} = F_{3\ell}$ for some $\ell \in \bbN$. Using \eqref{eq:rem_cont} it then follows that
    $$
    \frac12 F_{\mu_i+2} = \frac12 F_{3\ell+1} < \left\lfloor \phi \frac{F_{3\ell}}{2} + \phi^{-1} \right\rfloor = \lfloor \phi W_{i, 1} + \phi^{-1} \rfloor = W_{i, 2},
    $$
    also giving the left inequality of \eqref{ineq:k=2}. This fully establishes \eqref{ineq:k=2}.

    \textit{Step 3, $k \geq 3$.} Combining \eqref{ineq:k=1} and \eqref{ineq:k=2} with the recurrence relation \eqref{eq:recurrence} for $(W_{i, k})_k$ the desired claim follows.
\end{proof}

The number $\mu_i$ can be seen as a characteristic of $(W_{i, k})_k$, similar to $\eta_i$ above. We even have the following simple expression.

\begin{prop} \label{prop:mu_formula}
    For all $i \in \bbN$ it holds
    $$
    \mu_i = \lfloor\log_\phi(2 w_+(i))\rfloor.
    $$
\end{prop}

For the proof we need an auxiliary result.

\begin{lem} \label{lem:no_sol}
    For all $m, i \in \bbN$ it holds $\phi^m \neq 2 w_+(i)$.
\end{lem}
\begin{proof}
    For the sake of contradiction assume there are $m, i \in \bbN$ with $\phi^m=2w_+(i)$. Using \eqref{ineq:wp_wm_expressions} this would be equivalent to
    $$
    F_m \phi + F_{m-1} = \phi^m = 2w_+(i) = 2 \lfloor \phi i\rfloor \phi + 2(i-1),
    $$
    by irrationality of $\phi$ thus $F_m = 2 \lfloor \phi i\rfloor$ and $F_{m-1} = 2(i-1)$. However, no two consecutive terms of the Fibonacci sequences are both even (see the remark before Lemma \ref{lem:half_fib}), giving the desired contradiction.
\end{proof}

\begin{proof}[Proof of Proposition \ref{prop:mu_formula}]
    From Theorem \ref{thm:const_shift} we have $\frac12 F_{\mu_i+k} \leq W_{i, k} < \frac12 F_{\mu_i+k+1}$ for all $k \in \bbN$, which by \eqref{eq:binet} and Theorem \ref{thm:wythoff_basics} (i) can be rewritten a
    $$
    \phi^{\mu_i+k} - \left(-\frac1\phi\right)^{\mu_i+k} \leq 2 \left(w_+(i) \phi^k - w_-(i) \left(-\frac1\phi\right)^k \right) < \phi^{\mu_i+k+1} - \left(-\frac1\phi\right)^{\mu_i+k+1}.
    $$
    Dividing by $\phi^k$ and letting $k$ go to infinity we get $\phi^{\mu_i} \leq 2 w_+(i) \leq \phi^{\mu_i+1}$. By Lemma \ref{lem:no_sol} the equality cases cannot occur, so that even $\phi^{\mu_i} < 2 w_+(i) < \phi^{\mu_i+1}$ and the claim follows.
\end{proof}

For $n-k > \mu_i$ we define the \emph{dual array} (a different notion of duality 
than in \cite{Kim95})
\begin{align} \label{eq:dual}
    \begin{split}
        W_{i, n-k}^* & \coloneqq (-1)^k (F_{n-1} W_{i, k} - F_n W_{i, k-1}) \\
        & = \frac1{\sqrt5} \left(w_+^*(i) \phi^{n-k} - w_-^*(i) \left(-\frac1\phi\right)^{n-k}\right).
    \end{split}
\end{align}
In the second equality we have $w_+^*(i) = -w_-(i)$ and $w_-^*(i) = -w_+(i)$ as follows from \eqref{eq:binet} and Theorem \ref{thm:wythoff_basics} (i). Continuing $W_{i, k}$ also for $k \leq 0$ by requiring \eqref{eq:recurrence} for general $k \in \bbZ$, we may also write
\begin{align} \label{eq:extension}
    W_{i, n-k}^* = (-1)^{n-k} W_{i, -(n-k)},
\end{align}
compare with \cite{And14}. Some values for the dual array are given in the following table:
\begin{center}
	\begin{tabular}{r||c||c|c|c|c|c|c}
		$i$ & $\mu_i$ & $W^*_{i, \mu_i+1}$ & $W^*_{i, \mu_i+2}$ & $W^*_{i, \mu_i+3}$ & $W^*_{i, \mu_i+4}$ & $W^*_{i, \mu_i+5}$ & $W^*_{i, \mu_i+6}$ \\
		\hline
		$1$ & $2$ & $1$ & $2$ & $3$ & $5$ & $8$ & $13$ \\
		\hline
		$2$ & $5$ & $7$ & $11$ & $18$ & $29$ & $47$ & $76$ \\
		\hline
		$3$ & $5$ & $4$ & $6$ & $10$ & $16$ & $26$ & $42$ \\
		\hline
		$4$ & $6$ & $9$ & $15$ & $24$ & $39$ & $63$ & $102$ \\
		\hline
		$5$ & $7$ & $20$ & $32$ & $52$ & $84$ & $136$ & $220$ \\
		\hline
		$6$ & $7$ & $12$ & $19$ & $31$ & $50$ & $81$ & $131$ \\
		\hline
		$7$ & $8$ & $27$ & $44$ & $71$ & $115$ & $186$ & $301$
	\end{tabular}
\end{center}
The following bounds will be important in simplifying the terms of the sum in \eqref{eq:asymp_expansion}.

\begin{lem} \label{lem:dual_bounds}
    For $i, n, k \in \bbN$ with $n-k > \mu_i$ it holds
    $$
    F_{n-k-2} \leq W_{i, n-k}^* < F_{n-k}.
    $$
\end{lem}
\begin{proof}
    From the above table the assertion is easily verified for $i=1$ and $i=2$. Assume from now on that $i \geq 3$ and write $n-k = \mu_i+j$ for $j \in \bbN$. From \eqref{eq:W_fib} and \eqref{eq:extension} it holds $W_{i, \mu_i+j}^* = F_{\mu_i+j-1} \lfloor \phi i\rfloor - F_{\mu_i+j} (i-1)$, so that the desired inequality follows from
    \begin{align} \label{ineq:two_side_dual}
        F_{\mu_i+j} i < F_{\mu_i+j-1} (\lfloor \phi i \rfloor + 1) \quad \text{and} \quad F_{\mu_i+j-1} \lfloor \phi i \rfloor < F_{\mu_i+j} i.
    \end{align}
    We start by showing the right inequality.
    
    It is straight forward to verify that the inequality $1 + 5\phi^{2}i^2 < 4 (\phi^2 i-1)^2$ holds for all $i \geq 3$ (even for $i=2$). Using
    $$
    w_+(i) = \phi \lfloor\phi i\rfloor +i-1 > \phi i + i -1 = \phi^2i-1
    $$
    and $2 w_+(i) < \phi^{\mu_i+1}$ from Proposition \ref{prop:mu_formula} we then get for all $j \in \bbN$
    $$
    1+5\phi^2 i^2 < \phi^{2\mu_i+2} \leq \phi^{2\mu_i+2j},
    $$
    in particular
    $$
    (1+5\phi^2 i^2) \left(-\frac1\phi\right)^{\mu_i+j} < \phi^{\mu_i+j}.
    $$
    Using \eqref{eq:binet} and rearranging, this is equivalent to the inequality
    $$
    F_{\mu_i+j-1} < F_{\mu_i+j} \left(\frac1\phi + \frac1{\phi (\phi+2) i^2}\right).
    $$
    Applying the first inequality of Lemma \ref{lem:floor_ineq} gives the right inequality of \eqref{ineq:two_side_dual}.

    The left inequality can be shown in a similar manner. This time, for $i\geq3$ we start with
    $$
    2\sqrt5 \phi^3 i^2 - 1 < 4 (\phi^2i-1)^2 < 4 w_+(i)^2 < \phi^{2\mu_i+2} \leq \phi^{2\mu_i+2j},
    $$
    so that
    $$
    \left(1 - 2\sqrt5\phi^3 i^2\right) \left(-\frac1\phi\right)^{\mu_i+j} < \phi^{\mu_i+j},
    $$
    or by \eqref{eq:binet} and rearranging equivalently
    $$
    F_{\mu_i+j} \left(\frac1\phi - \frac1{2\phi^3 i^2}\right) < F_{\mu_i+j-1}.
    $$
    Now using the second inequality of Lemma \ref{lem:floor_ineq} we get the left inequality of \eqref{ineq:two_side_dual}. This finishes the proof.
\end{proof}

In particular we obtain
\begin{align} \label{ineq:in_range}
    0 < W_{i, n-k}^* < F_n
\end{align}
whenever $n-k > \mu_i$.

\section{Asymptotics} \label{sec:asymp}

In this section we present the proof of Theorem \ref{thm:asymp} and a closed-form expression of the constant $C$ in terms of a certain Dedekind zeta function.

\subsection{Approximations}

Note that, using \eqref{eq:dual}, \eqref{ineq:in_range} as well as the periodicity and symmetry of $f$, we have
\begin{align} \label{eq:introduce_dual}
    \frac{f(F_{n-1} W_{i, k}/F_n)}{\left|{\sin(\pi F_{n-1} W_{i, k}/F_n)}\right|^\sigma} = \frac{f(W^*_{i, n-k}/F_n)}{{\sin(\pi W^*_{i, n-k}/F_n)}^\sigma}.
\end{align}
Using the asymptotics of $W_{i, k}$ and $W_{i, n-k}^*$ we approximate
$$
\frac1{F_n^\sigma} \frac{f(W_{i, k}/F_n)}{{\sin(\pi W_{i, k}/F_n)}^\sigma} \frac{f(W^*_{i, n-k}/F_n)}{{\sin(\pi W^*_{i, n-k}/F_n)}^\sigma} \approx \frac{f(0) }{\pi^\sigma W_{i, k}^\sigma} \frac{f(w_+^*(i) \phi^{-k})}{{\sin(\pi w_+^*(i) \phi^{-k})}^\sigma}
$$
for small $k$ and
\begin{align*}
    \frac1{F_n^\sigma} \frac{f(W_{i, k}/F_n)}{{\sin(\pi W_{i, k}/F_n)}^\sigma} \frac{f(W^*_{i, n-k}/F_n)}{{\sin(\pi W^*_{i, n-k}/F_n)}^\sigma}
    \approx \frac{f(w_+(i) \phi^{-(n-k)})}{{\sin(\pi w_+(i) \phi^{-(n-k)})}^\sigma} \frac{f(0)}{ \pi^\sigma (W_{i, n-k}^*)^\sigma}
\end{align*}
for large $k$. We will bound the error in these approximations.

\begin{lem} \label{lem:approx_quality}
    Let $\sigma > 1$ and $f:[0, 1) \rightarrow \bbR$ be a $1$-periodic, $\alpha$-Hölder continuous function for some $0 < \alpha \leq 1$. For $i \in \bbN$ let $\mu_i$ be as in Theorem \ref{thm:const_shift} and let $n, k \in \bbN$ be such that $k < n - \mu_i$. There are constants $A = A(\sigma, f, \alpha), A^*=A^*(\sigma, f, \alpha) > 0$ such that
    \begin{align*}
        \left|\frac1{F_n^\sigma} \frac{f(W_{i, k}/F_n)}{{\sin(\pi W_{i, k}/F_n)}^\sigma} \frac{f(W^*_{i, n-k}/F_n)}{{\sin(\pi W^*_{i, n-k}/F_n)}^\sigma} - \frac{f(0) }{\pi^\sigma W_{i, k}^\sigma} \frac{f(w_+^*(i) \phi^{-k})}{{\sin(\pi w_+^*(i) \phi^{-k})}^\sigma}\right| \leq \frac{A}{\phi^{\alpha (n-k)} i^{\sigma - \alpha}}
    \end{align*}
    and
    \begin{align*}
        & \left|\frac1{F_n^\sigma} \frac{f(W_{i, k}/F_n)}{{\sin(\pi W_{i, k}/F_n)}^\sigma} \frac{f(W^*_{i, n-k}/F_n)}{{\sin(\pi W^*_{i, n-k}/F_n)}^\sigma}
        - \frac{f(w_+(i) \phi^{-(n-k)})}{{\sin(\pi w_+(i) \phi^{-(n-k)})}^\sigma} \frac{f(0)}{ \pi^\sigma (W_{i, n-k}^*)^\sigma}\right| \\
        \leq ~ & \frac{A^*}{\phi^{\alpha k} i^{\sigma-\alpha}}.
    \end{align*}
\end{lem}
\begin{proof}
    We first make some preparation. Every time we write $x \lesssim y$ or $x \gtrsim y$ we mean the corresponding inequality up to a constant independent of $n, k, i$ (but possibly depending on $\sigma, f, \alpha$). The inequality
    \begin{align} \label{ineq:quot_bound}
        0 < -\frac{w_-(i)}{w_+(i)} \leq \phi^{-2},
    \end{align}
    easily derived from \eqref{ineq:wp_wm_expressions}, will be used multiple times. From Proposition \ref{prop:mu_formula} we have $n-k > \mu_i = \lfloor \log_\phi(2 w_+(i))\rfloor$ and thus even
    \begin{align} \label{ineq:n_k_i}
        \phi^{n-k} > 2 w_+(i) \gtrsim i.
    \end{align}
    Furthermore, from \eqref{mu_i} we note that $n > k + \mu_i \geq 1+2$, so that we may assume $n \geq 4$. By Lemma \ref{lem:dual_bounds} in particular
    $$
    \frac{W_{i, n-k}^*}{F_n} < \frac{F_{n-k}}{F_n} \leq \frac{F_{n-1}}{F_n} \leq \frac23,
    $$
    similar to
    $$
    \frac{W_{i, k}}{F_n} < \frac12
    $$
    by the defining property of $\mu_i$. We thus may (and will) apply the inequalities of Lemma \ref{lem:calc_ineqs} for $x, y \in \{W_{i, k}/F_n, W_{i, n-k}^*/F_n\}$.
    
    We will frequently use that $f$ is bounded as a periodic and continuous function, in particular $|f(t)| \leq \|f\|_\infty \lesssim 1$ for all $t \in [0, 1)$ with the supremum norm.

    We will show the first inequality, the second can be dealt with in an analogous way. First and foremost, by the triangle inequality we can write
    \begin{align*}
        & \left|\frac{f(W_{i, k}/F_n)}{F_n^\sigma {\sin(\pi W_{i, k}/F_n)}^\sigma} \frac{f(W^*_{i, n-k}/F_n)}{{\sin(\pi W^*_{i, n-k}/F_n)}^\sigma} - \frac{f(0) }{\pi^\sigma W_{i, k}^\sigma} \frac{f(w_+^*(i) \phi^{-k})}{{\sin(\pi w_+^*(i) \phi^{-k})}^\sigma}\right| \\
        \leq ~ & |\text{I}| + |\text{II}| + |\text{III}| + |\text{IV}|,
    \end{align*}
    where the symbols $|\text{I}|, |\text{II}|, |\text{III}|, |\text{IV}|$ denote the expressions
    \begin{align*}
        |\text{I}| & = \left| \frac{f(W_{i, k}/F_n)}{F_n^\sigma{\sin(\pi W_{i, k}/F_n)}^\sigma} \frac{f(W^*_{i, n-k}/F_n)}{{\sin(\pi W^*_{i, n-k}/F_n)}^\sigma} - \frac{f(W_{i, k}/F_n)}{F_n^\sigma{\sin(\pi W_{i, k}/F_n)}^\sigma} \frac{f(W^*_{i, n-k}/F_n)}{\sin(\pi w_+^*(i) \phi^{-k})^\sigma} \right|, \\
        |\text{II}| & = \left| \frac{f(W_{i, k}/F_n)}{F_n^\sigma{\sin(\pi W_{i, k}/F_n)}^\sigma} \frac{f(W^*_{i, n-k}/F_n)}{\sin(\pi w_+^*(i) \phi^{-k})^\sigma} - \frac{f(W_{i, k}/F_n)}{\pi^\sigma W_{i, k}^\sigma} \frac{f(W^*_{i, n-k}/F_n)}{\sin(\pi w_+^*(i) \phi^{-k})^\sigma} \right|, \\
        |\text{III}| & = \left| \frac{f(W_{i, k}/F_n)}{\pi^\sigma W_{i, k}^\sigma} \frac{f(W^*_{i, n-k}/F_n)}{\sin(\pi w_+^*(i) \phi^{-k})^\sigma} - \frac{f(0)}{\pi^\sigma W_{i, k}^\sigma} \frac{f(W^*_{i, n-k}/F_n)}{\sin(\pi w_+^*(i) \phi^{-k})^\sigma} \right|, \\
        |\text{IV}| & = \left| \frac{f(0)}{\pi^\sigma W_{i, k}^\sigma} \frac{f(W^*_{i, n-k}/F_n)}{\sin(\pi w_+^*(i) \phi^{-k})^\sigma} - \frac{f(0)}{\pi^\sigma W_{i, k}^\sigma} \frac{f(w_+^*(i) \phi^{-k})}{\sin(\pi w_+^*(i) \phi^{-k})^\sigma} \right|.
    \end{align*}
    These four terms will be bounded individually.
    
    For the first term we have by Lemma \ref{lem:calc_ineqs} (i)
    \begin{align*}
        |\text{I}| & \lesssim \frac{1}{W_{i, k}^\sigma} \left|\frac{1}{\sin(\pi W_{i, n-k}^*/F_n)^\sigma} - \frac{1}{\sin(\pi w_+^*(i) \phi^{-k})^\sigma}\right|.
    \end{align*}
    Noting that by \eqref{lin_growth} and \eqref{ineq:quot_bound}
    $$
    \frac1{W_{i, k}} = \frac1{w_+(i) \phi^k} \frac{\sqrt5}{1 - \frac{w_-(i)}{w_+(i)} \frac{(-1)^k}{\phi^{2k}}} \lesssim \frac{1}{\phi^k i} \frac{\sqrt5}{1 - \phi^{-2} \cdot \frac1{\phi^2}} \lesssim \frac{1}{\phi^k i},
    $$
    by \eqref{ineq:w_minus} and \eqref{ineq:n_k_i}
    \begin{align*}
        \frac{F_n}{W_{i, n-k}^*} & = \frac{\phi^n}{\phi^{n-k}} \frac{1-(-1)^n \phi^{-2n}}{w_+^*(i)-w_-^*(i) \frac{(-1)^{n-k}}{\phi^{2(n-k)}} } = \phi^k \frac{1-(-1)^n \phi^{-2n}}{-w_-(i)+\frac{w_+(i)}{\phi^{n-k}} \frac{(-1)^{n-k}}{\phi^{n-k}} } \leq \phi^k\frac{1+\phi^{-10}}{\frac1{\phi^2}-\frac12 \frac1{\phi^3}} \\
        & \lesssim \phi^k.
    \end{align*}
    and by Lemma \ref{lem:calc_ineqs} (ii), \eqref{ineq:w_minus}, \eqref{ineq:quot_bound} and \eqref{ineq:n_k_i}
    \begin{align*}
        & \left|\frac{1}{\sin\left(\pi \frac{W_{i, n-k}^*}{F_n}\right)^\sigma} - \frac{1}{\sin\left(\pi \frac{w_+^*(i)}{\phi^{k}}\right)^\sigma}\right| \leq \left(\frac{F_n}{W_{i, n-k}^*} + \frac{\phi^k}{w_+^*(i)}\right)^{\sigma-1} \left|\frac{F_n}{W_{i, n-k}^*} - \frac{\phi^k}{w_+^*(i)}\right| \\
        \lesssim ~ & \phi^{(\sigma-1)k} \left|\frac{\phi^n}{w_+^*(i) \phi^{n-k}} \frac{1-(-1)^n\phi^{-2n}}{1-\frac{w_-^*(i)}{w_+^*(i)} (-1)^{n-k} \phi^{-2(n-k)}} - \frac{\phi^k}{w_+^*(i)}\right| \\
        \lesssim ~ & \phi^{(\sigma-1)k} \cdot \phi^k \left|\frac{\frac{w_-^*(i)}{w_+^*(i)} (-1)^{n-k} \phi^{-2(n-k)}-(-1)^n \phi^{-2n}}{1-\frac{w_-^*(i)}{w_+^*(i)} (-1)^{n-k} \phi^{-2(n-k)}}\right| \\
        \lesssim ~ & w_+(i) \phi^{(2+\sigma) k - 2n} \left|\frac{1-\frac{w_-(i)}{w_+(i)} (-1)^k \phi^{-2k}}{-w_-(i)+\frac{w_+(i)}{\phi^{n-k}} \frac{(-1)^{n-k}}{\phi^{n-k}}}\right| \lesssim \frac{1+\phi^{-2} \cdot \phi^{-2}}{\phi^{-2} - \frac12 \frac1{\phi^3}} \frac{i}{\phi^{2n-(\sigma+2)k}} \lesssim \frac{i}{\phi^{2n-(\sigma+2)k}},
    \end{align*}
    we thus get once more by \eqref{ineq:n_k_i} the desired
    $$
    |\text{I}| \lesssim \left(\frac{1}{\phi^k i}\right)^\sigma \frac{i}{\phi^{2n-(\sigma+2)k}} = \frac1{\phi^{2(n-k)} i^{\sigma-1}} \lesssim \textcolor{black}{\frac1{\phi^{n-k} i^\sigma}} \lesssim \frac1{\phi^{\alpha(n-k)} i^{\sigma-\alpha}}.
    $$
    We will abbreviate the argument a bit for the other terms. For the second term we have, using Lemma \ref{lem:calc_ineqs} (iii) and \eqref{ineq:n_k_i},
    \begin{align*}
        |\text{II}| & \lesssim \left|\frac{1}{F_n^\sigma \sin(\pi W_{i, k}/F_n)^\sigma}-\frac{1}{\pi^\sigma W_{i, k}^\sigma}\right| \phi^{\sigma k} = \frac{\phi^{\sigma k}}{\pi^\sigma W_{i, k}^\sigma} \left|\left(\frac{\pi W_{i, k}/F_n}{\sin(\pi W_{i, k}/F_n)}\right)^\sigma-1\right| \\
        & \lesssim \frac{\phi^{\sigma k}}{\phi^{\sigma k} i^\sigma} \left(\frac{W_{i, k}}{F_n}\right)^2 \lesssim \frac1{i^\sigma} \frac{i^2}{\phi^{2(n-k)}} = \frac{1}{\phi^{\alpha(n-k)} \phi^{(2-\alpha)(n-k)} i^{\sigma-2}} \lesssim \frac1{\phi^{\alpha (n-k)} i^{\sigma-\alpha}}.
    \end{align*}
    For the third term we have, by $\alpha$-Hölder continuity of $f$,
    \begin{align*}
        |\text{III}| & \lesssim \left(\frac1{\phi^k i}\right)^\sigma \phi^{\sigma k} \left|f\left(\frac{W_{i, k}}{F_n}\right)-f(0)\right| \lesssim \frac1{i^\sigma} \left(\frac{W_{i, k}}{F_n}\right)^\alpha \lesssim \frac1{i^\sigma} \left(\frac{i}{\phi^{n-k}}\right)^\alpha \\
        & = \frac1{\phi^{n-k} i^{\sigma-\alpha}} \lesssim \frac1{\phi^{\alpha(n-k)} i^{\sigma-\alpha}}.
    \end{align*}
    Lastly, for the fourth term we similarly have
    \begin{align*}
        |\text{IV}| & \lesssim \left(\frac1{\phi^k i}\right)^\sigma \phi^{\sigma k} \left|f\left(\frac{W_{i, n-k}^*}{F_n}\right) - f\left(\frac{w_+^*(i)}{\phi^k}\right)\right| \lesssim \frac1{i^\sigma} \left|\frac{W_{i, n-k}^*}{F_n} - \frac{w_+^*(i)}{\phi^k}\right|^\alpha \\
        & = \frac1{i^\sigma} \left|\frac{-w_-^*(i)}{\phi^{2n-k}} \frac{1 - \frac{w_+^*(i)}{w_-^*(i)} (-1)^{k} \phi^{-2k}}{1-(-1)^n \phi^{-2n}}\right|^\alpha \lesssim \frac1{i^\sigma} \left(\frac{i}{\phi^{2n-k}}\right)^\alpha = \frac1{\phi^{\alpha(2n-k)} i^{\sigma-\alpha}} \\
        & \lesssim \frac1{\phi^{\alpha(n-k)} i^{\sigma-\alpha}}.
    \end{align*}
    It now remains to combine $|\text{I}| + |\text{II}| + |\text{III}| + |\text{IV}|$ to get the desired bound.
\end{proof}

We will further approximate
\begin{align*}
    \frac{f(0)}{\pi^\sigma W_{i, k}^\sigma} \frac{f(w_+^*(i) \phi^{-k})}{{\sin(\pi w_+^*(i) \phi^{-k})}^\sigma}
    \approx \frac{f(0)^2 5^{\sigma/2}}{\pi^{2\sigma} \eta_i^\sigma}
    \approx \frac{f(w_+(i) \phi^{-(n-k)})}{{\sin(\pi w_+(i) \phi^{-(n-k)})}^\sigma} \frac{f(0)}{\pi^\sigma (W_{i, n-k}^*)^\sigma}.
\end{align*}
The next result shows that the errors in these approximations decay exponentially.

\begin{lem} \label{lem:exp_decay}
    Let $\sigma > 1$ and let $f:[0, 1) \rightarrow \bbR$ be $1$-periodic and $\alpha$-Hölder continuous for some $0<\alpha\leq 1$. Let $i \in \bbN$. There is a constant $B = B(\sigma, f, \alpha) > 0$ such that for all $k \in \bbN$ it holds
    $$
    \left|\frac{f(0)}{\pi^\sigma W_{i, k}^\sigma} \frac{f(w_+^*(i) \phi^{-k})}{{\sin(\pi w_+^*(i) \phi^{-k})}^\sigma} - \frac{f(0)^2 5^{\sigma/2}}{\pi^{2\sigma} \eta_i^\sigma}\right| \leq \frac{B}{\phi^{\alpha k} i^\sigma}.
    $$
    Also, there is a constant $B^* = B^*(\sigma, f, \alpha) > 0$ such that for all $j \in \bbN$ it holds
    $$
    \left|\frac{f(w_+(i) \phi^{-(\mu_i+j)})}{{\sin(\pi w_+(i) \phi^{-(\mu_i+j)})}^\sigma} \frac{f(0)}{\pi^\sigma (W_{i, \mu_i+j}^*)^\sigma} - \frac{f(0)^2 5^{\sigma/2}}{\pi^{2\sigma} \eta_i^\sigma}\right| \leq \frac{B^*}{\phi^{\alpha (\mu_i+j)} i^{\sigma-\alpha}}.
    $$
\end{lem}
\begin{proof}
    The proof is similar to the one of the previous lemma and we will again only discuss the first expression in detail. Bound
    $$
    \left|\frac{f(0)}{\pi^\sigma W_{i, k}^\sigma} \frac{f(w_+^*(i) \phi^{-k})}{{\sin(\pi w_+^*(i) \phi^{-k})}^\sigma} - \frac{f(0)^2 5^{\sigma/2}}{\pi^{2\sigma} \eta_i^\sigma}\right| \leq |\text{I}|+|\text{II}|+|\text{III}|,
    $$
    where now
    \begin{align*}
        |\text{I}| & = \left|\frac{f(0)}{\pi^\sigma W_{i, k}^\sigma} \frac{f(w_+^*(i) \phi^{-k})}{{\sin(\pi w_+^*(i) \phi^{-k})}^\sigma} - \frac{f(0) 5^{\sigma/2}}{\pi^\sigma w_+(i)^\sigma \phi^{\sigma k}} \frac{f(w_+^*(i) \phi^{-k})}{{\sin(\pi w_+^*(i) \phi^{-k})}^\sigma}\right|, \\
        |\text{II}| & = \left|\frac{f(0) 5^{\sigma/2}}{\pi^\sigma w_+(i)^\sigma \phi^{\sigma k}} \frac{f(w_+^*(i) \phi^{-k})}{{\sin(\pi w_+^*(i) \phi^{-k})}^\sigma} - \frac{f(0) 5^{\sigma/2}}{\pi^\sigma w_+(i)^\sigma \phi^{\sigma k}} \frac{f(0)}{{\sin(\pi w_+^*(i) \phi^{-k})}^\sigma}\right|, \\
        |\text{III}| & = \left|\frac{f(0) 5^{\sigma/2}}{\pi^\sigma w_+(i)^\sigma \phi^{\sigma k}} \frac{f(0)}{{\sin(\pi w_+^*(i) \phi^{-k})}^\sigma} - \frac{f(0) 5^{\sigma/2}}{\pi^\sigma w_+(i)^\sigma \phi^{\sigma k}} \frac{f(0)}{\pi^\sigma w_+^*(i)^\sigma \phi^{-\sigma k}}\right|,
    \end{align*}
    recalling that $w_+(i) w_+^*(i) = \eta_i$. These terms can be treated individually. For the first one we have, using Lemma \ref{lem:calc_ineqs} (iv),
    \begin{align*}
        |\text{I}| & \lesssim \frac1{\phi^{-\sigma k}} \left|\frac{1}{W_{i, k}^\sigma} - \frac{5^{\sigma/2}}{w_+(i)^\sigma \phi^{\sigma k}}\right| \lesssim \phi^{\sigma k} \left|\left(\frac1{w_+(i) \phi^k - w_-(i) \left(-\frac1\phi\right)^k}\right)^\sigma - \frac1{w_+(i)^\sigma \phi^{\sigma k}}\right| \\
        & \lesssim \frac1{i^\sigma} \left|\left(\frac{1}{1- \frac{w_-(i)}{w_+(i)} (-1)^k \phi^{-2k}}\right)^\sigma - 1\right| \lesssim \frac1{i^\sigma} \left|\frac{w_-(i)}{w_+(i)} \frac1{\phi^{2k}}\right| \lesssim \frac1{\phi^{2k} i^{\sigma+1}} \lesssim \frac1{\phi^{\alpha k} i^\sigma}.
    \end{align*}
    For the second term, by $\alpha$-Hölder continuity of $f$ it holds
    \begin{align*}
        |\text{II}| & \lesssim \frac1{w_+(i)^\sigma \phi^{\sigma k}} \frac1{\phi^{-\sigma k}} \left|f\left(\frac{w_+^*(i)}{\phi^k}\right) - f(0)\right| \lesssim \frac1{i^\sigma} \left(\frac{w_+^*(i)}{\phi^k}\right)^\alpha \lesssim \frac1{\phi^{\alpha k} i^\sigma}.
    \end{align*}
    Finally, for the third term we have by Lemma \ref{lem:calc_ineqs} (iii)
    \begin{align*}
        |\text{III}| & \lesssim \frac1{w_+(i)^\sigma \phi^{\sigma k}} \left|\frac1{\sin(\pi w_+^*(i)\phi^{-k})^\sigma} - \frac1{\pi^\sigma w_+^*(i)^\sigma \phi^{-\sigma k}}\right| \\
        & \lesssim \frac1{i^\sigma} \left|\left(\frac{\pi w_+^*(i)\phi^{-k}}{\sin(\pi w_+^*(i)\phi^{-k})}\right)^\sigma-1\right| \lesssim \frac1{i^\sigma} \frac1{\phi^{2k}} \lesssim \frac1{\phi^{\alpha k} i^\sigma}.
    \end{align*}
    The terms $|\text{I}|+|\text{II}|+|\text{III}|$ together then yield the claim.
\end{proof}

The previous lemma shows that the sums
$$
\delta_{f, \sigma}(i) \coloneqq \sum_{k=1}^\infty \left(\frac{f(0)}{\pi^\sigma W_{i, k}^\sigma} \frac{f(w_+^*(i) \phi^{-k})}{{\sin(\pi w_+^*(i) \phi^{-k})}^\sigma} - \frac{f(0)^2 5^{\sigma/2}}{\pi^{2\sigma} \eta_i^\sigma}\right) \lesssim \frac1{i^\sigma}
$$
and
$$
\delta_{f, \sigma}^*(i) \coloneqq \sum_{j=1}^\infty \left(\frac{f(w_+(i) \phi^{-(\mu_i+j)})}{{\sin(\pi w_+(i) \phi^{-(\mu_i+j)})}^\sigma} \frac{f(0)}{\pi^\sigma (W_{i, \mu_i+j}^*)^\sigma} - \frac{f(0)^2 5^{\sigma/2}}{\pi^{2\sigma} \eta_i^\sigma}\right) \lesssim \frac1{i^{\sigma-\alpha}}
$$
converge with geometric convergence rate. In particular, truncating the sum after the first $n$ terms, compared to the infinite sum this leaves an error of order $O(\phi^{-\alpha n} i^{-\sigma})$ and $O(\phi^{-\alpha n} i^{-(\sigma-\alpha)})$ respectively.

\subsection{Proof of Theorem \ref{thm:asymp}}

We are now in a position to prove our main result.

\begin{proof}[Proof of Theorem \ref{thm:asymp}]
    Any $\alpha$-Hölder continuous, symmetric, $1$-periodic function $f: [0, 1) \rightarrow \bbR$ is also $\beta$-Hölder continuous for all $0 < \beta \leq \alpha$. This, in particular, holds then for $\beta \coloneq \min\{\alpha, \sigma -1\}$. Also set $\gamma \coloneqq \beta/2$ so that in total $0 < \gamma < \beta \leq \alpha$. We will apply the previous two lemmas with the exponent $\beta$ (and at one instance with $\gamma$).

    By symmetry of $f$ we may write
    \begin{align*}
        \Sigma & \coloneqq \frac1{F_n^\sigma} \sum_{m = 1}^{F_n-1} \frac{f(m/F_n)}{\left|{\sin(\pi m/F_n)}\right|^\sigma} \frac{f(F_{n-1}m/F_n)}{\left|{\sin(\pi F_{n-1}m/F_n)}\right|^\sigma} \\
        & = \frac2{F_n^\sigma} \sum_{1 \leq m < F_n/2} \frac{f(m/F_n)}{\left|{\sin(\pi m/F_n)}\right|^\sigma} \frac{f(F_{n-1}m/F_n)}{\left|{\sin(\pi F_{n-1}m/F_n)}\right|^\sigma} + O(\phi^{-\sigma n}),
    \end{align*}
    where the error is $0$ if $F_n$ is odd and (by periodicity of $f$)
    $$
    \frac1{F_n^\sigma} \left(\frac{f(1/2)}{\left|{\sin(\pi/2)}\right|^\sigma}\right)^2 = O(\phi^{-\sigma n})
    $$
    if $F_n$ is even. Partition now the index set $\{m \in \bbN: m < F_n/2\}$ according to the Wythoff array, also using \eqref{eq:introduce_dual},
    $$
    \Sigma = 2\sum_{\substack{i \in \bbN: \\ n - \mu_i > 1}} \sum_{k=1}^{n-\mu_i-1} \frac1{F_n^\sigma} \frac{f(W_{i, k}/F_n)}{\sin(\pi W_{i, k}/F_n)^\sigma} \frac{f(W_{i, n-k}^*/F_n)}{\sin(\pi W_{i, n-k}^*/F_n)^\sigma} + O(\phi^{-\sigma n}).
    $$
    Splitting up the inner sum into $1 \leq k < \frac n2$ and (the possibly vacuous) $\frac n2 \leq k < n-\mu_i$ and applying Lemmas \ref{lem:approx_quality} and \ref{lem:exp_decay} we get
    \begin{align*}
        \Sigma = ~& 2\sum_{\substack{i \in \bbN: \\ n - \mu_i > 1}} \Bigg[ \sum_{1 \leq k < \frac n2} \left(\frac{f(0)}{\pi^\sigma W_{i, k}^\sigma} \frac{f(w_+^*(i) \phi^{-k})}{\sin(\pi w_+^*(i) \phi^{-k})^\sigma} + O\left(\frac1{\phi^{\beta n/2} i^{\sigma-\beta}}\right)\right) \\
        & ~~~~~~~~~~~ + \sum_{\frac n2 \leq k < n-\mu_i} \left(\frac{f(w_+(i) \phi^{-(n-k)})}{{\sin(\pi w_+(i) \phi^{-(n-k)})}^\sigma} \frac{f(0)}{ \pi^\sigma (W_{i, n-k}^*)^\sigma} + O\left(\frac1{\phi^{\beta n/2} i^{\sigma-\beta}}\right)\right) \Bigg] \\
        & + O(\phi^{-\sigma n}) \\
        = ~& 2\sum_{\substack{i \in \bbN \\ n-\mu_i > 1}} \Bigg[ (n-\mu_i-1) \frac{f(0)^2 5^{\sigma/2}}{\pi^{2\sigma} \eta_i^\sigma} \\
        & ~~~~~~~~~~~ + \sum_{1 \leq k < \frac n2} \left(\frac{f(0)}{\pi^\sigma W_{i, k}^\sigma} \frac{f(w_+^*(i) \phi^{-k})}{{\sin(\pi w_+^*(i) \phi^{-k})}^\sigma} - \frac{f(0)^2 5^{\sigma/2}}{\pi^{2\sigma} \eta_i^\sigma}\right) \\
        & ~~~~~~~~~~~ + \sum_{1 \leq j < \frac n2 - \mu_i} \left(\frac{f(w_+(i) \phi^{-(\mu_i+j)})}{{\sin(\pi w_+(i) \phi^{-(\mu_i+j)})}^\sigma} \frac{f(0)}{\pi^\sigma (W_{i, \mu_i+j}^*)^\sigma} - \frac{f(0)^2 5^{\sigma/2}}{\pi^{2\sigma} \eta_i^\sigma}\right) \\
        & ~~~~~~~~~~~ + O\left(\frac{n}{\phi^{\beta n/2} i^{\sigma-\beta}}\right) \Bigg] + O(\phi^{-\sigma n}) \\
        = ~& 2\sum_{\substack{i \in \bbN \\ n-\mu_i > 1}} \Bigg[ (n-\mu_i-1) \frac{f(0)^2 5^{\sigma/2}}{\pi^{2\sigma} \eta_i^\sigma} + \left(\delta_{f, \sigma}(i) + O\left(\frac1{\phi^{\beta n/2} i^\sigma}\right)\right) \\
        & ~~~~~~~~~~~ + \left(\delta_{f, \sigma}^*(i) + O\left(\frac1{\phi^{\beta n/2} i^{\sigma-\beta}}\right)\right) + O\left(\frac{n}{\phi^{\beta n/2} i^{\sigma-\beta}}\right) \Bigg] + O(\phi^{-\sigma n}) \\
        = ~& \frac{f(0)^2 5^{\sigma/2}}{\pi^{2\sigma}} \left(\sum_{\substack{i \in \bbN \\ n-\mu_i > 1}} \frac1{\eta_i^\sigma}\right) n + \sum_{\substack{i \in \bbN \\ n-\mu_i > 1}} \left(\delta_{f, \sigma}(i) + \delta_{f, \sigma}^*(i) - \frac{f(0)^2 5^{\sigma/2}}{\pi^{2\sigma}} \frac{\mu_i+1}{\eta_i^\sigma}\right) \\
        & + \sum_{\substack{i \in \bbN \\ n-\mu_i > 1}}O\left(\frac{n}{\phi^{\beta n/2} i^{\sigma-\beta}}\right) + O(\phi^{-\sigma n}).
    \end{align*}
    It remains to look at the errors we introduce when replacing the finite sums by infinite ones. Using Proposition \ref{prop:mu_formula} we have the equivalence
    $$
    n - \mu_i > 1 \quad \Leftrightarrow \quad \phi^n > 2\phi w_+(i)
    $$
    and by \eqref{lin_growth} thus the implications
    $$
    i < \frac{\phi^n}{12} \quad \Rightarrow \quad n-\mu_i > 1 \quad \Rightarrow \quad i < \frac{\phi^n}{5}.
    $$
    Since $\eta_i \asymp i$, going from the finite sum
    $$
    \sum_{\substack{i \in \bbN \\ n-\mu_i > 1}} \frac1{\eta_i^\sigma} \geq \sum_{i = 1}^{\lfloor\phi^n/12 \rfloor} \frac1{\eta_i^\sigma}
    $$
    to the infinite sum
    $$
    \sum_{i = 1}^\infty \frac1{\eta_i^\sigma}
    $$
    introduces an error of order $O\left(\phi^{-(\sigma-1)n}\right)$. Since similarly $\delta_{f, \sigma}(i) \lesssim i^{-\sigma}$, the same error gets introduced in the sum over this quantity. Since $\mu_i \asymp \log i$, for the sum over $(\mu_i+1)/\eta_i^\sigma$ we get an error of order $O\left(n/\phi^{(\sigma-1)n}\right)$. As for the sum over $\delta^*_{f, \sigma}(i)$, where instead of the exponent $\beta$ we will use $\gamma$, we have
    $$
    \delta^*_{f, \sigma}(i) \lesssim \frac1{i^{\sigma-\gamma}} = \frac1{i^{\sigma - \beta + \gamma}} \leq \frac1{i^{1+\gamma}}
    $$
    and thus we get an error of the order
    $$
    \sum_{\substack{i \in \bbN \\ n-\mu_i \leq 1}} \delta^*_{f, \sigma}(i) \lesssim \sum_{\substack{i \in \bbN \\ i \geq \phi^n/5}} \frac1{i^{1+\gamma}} \lesssim \phi^{-\gamma n} = \phi^{-\beta n/2}.
    $$
    Lastly, since $\sigma-\beta \geq 1$ we have
    $$
    \sum_{\substack{i \in \bbN \\ n-\mu_i > 1}} \frac1{i^{\sigma-\beta}} \leq \sum_{\substack{i \in \bbN \\ n-\mu_i > 1}} \frac1{i} \lesssim \log(\phi^n) \lesssim n.
    $$
    Collecting all errors together we arrive at
    \begin{align*}
        & O\left(\frac{n}{\phi^{(\sigma-1)n}}\right) + O\left(\frac{1}{\phi^{(\sigma-1)n}}\right) + O\left(\frac{1}{\phi^{\beta n/2}}\right) + O\left(\frac{n}{\phi^{(\sigma-1)n}}\right) \\
        & + O\left(\frac{n^2}{\phi^{\beta n/2}}\right) + O\left(\frac{1}{\phi^{\sigma n}}\right) \\
        = ~ & O\left(\frac{n^2}{\phi^{\beta n/2}}\right).
    \end{align*}
    All together we obtain
    \begin{align*}
        \Sigma = ~ & 2\frac{f(0)^2 5^{\sigma/2}}{\pi^{2\sigma}} \left(\sum_{i=1}^\infty \frac1{\eta_i^\sigma}\right) n + 2\sum_{i=1}^\infty \left(\delta_{f, \sigma}(i) + \delta_{f, \sigma}^*(i) - \frac{f(0)^2 5^{\sigma/2}}{\pi^{2\sigma}} \frac{\mu_i+1}{\eta_i^\sigma}\right) \\
        & + O\left(\frac{n^2}{\phi^{\beta n/2}}\right),
    \end{align*}
    so that the constants $C$ and $D$ are given by
    \begin{align} \label{C}
        C = 2\frac{f(0)^2 5^{\sigma/2}}{\pi^{2\sigma}} \sum_{i=1}^\infty \frac1{\eta_i^\sigma}
    \end{align}
    and
    \begin{align} \label{D}
        D = 2\sum_{i=1}^\infty \left(\delta_{f, \sigma}(i) + \delta_{f, \sigma}^*(i) - \frac{f(0)^2 5^{\sigma/2}}{\pi^{2\sigma}} \frac{\mu_i+1}{\eta_i^\sigma}\right),
    \end{align}
    finishing the proof.
\end{proof}

Using \eqref{C} and \eqref{D} we can determine the constants $C$ and $D$ to arbitrary precision. For example, for the sum in \eqref{eq:wce_expression} with $\sigma = 2, 4, 6$ we get
\begin{align*}
    & \frac1{F_n^2} \sum_{m=1}^{F_n-1} \frac1{\sin(\pi m / F_n)^2 \sin(\pi F_{n-1} m/F_n)^2} \\
    & \quad\quad = 0.119256958\dots n - 0.075555555\ldots + O\left(\frac{n^2}{\phi^{n/2}}\right), \\
    & \frac1{F_n^4} \sum_{m=1}^{F_n-1} \frac{2+4\cos(\pi m/F_n)^2}{\sin(\pi m/F_n)^4} \frac{2+4\cos(\pi F_{n-1}m/F_n)^2}{\sin(\pi F_{n-1}m/F_n)^4} \\
    & \quad\quad = 0.190811134\dots n + 0.174222222\ldots+O\left(\frac{n^2}{\phi^{n/2}}\right), \\
    & \frac1{F_n^6} \sum_{m=1}^{F_n-1} \frac{16+88\cos(\pi m/F_n)^2+16\cos(\pi m/F_n)^4}{\sin(\pi m/F_n)^6} \\
    & \phantom{\frac1{F_n^6}} ~~~~~~~ \times \frac{16+88\cos(\pi F_{n-1}m/F_n)^2+16\cos(\pi F_{n-1}m/F_n)^4}{\sin(\pi F_{n-1}m/F_n)^6} \\
    & \quad\quad = 3.896181633\dots n+0.369674379\ldots+O\left(\frac{n^2}{\phi^{n/2}}\right).
\end{align*}
In the remaining parts of this paper we discuss how one can determine closed-form expressions for these constants. The last section will also show that the error term can probably be improved for such sums to $O(n/\phi^{2n})$ under suitable assumptions on $f$.

\subsection{The number field $\bbQ(\sqrt5)$ and its Dedekind zeta function}


Number theory, in a simplified way, is concerned with the arithmetic properties of domains of numbers, in particular the structure of its prime elements. While classically this was restricted to the integers $\bbZ$, it was realized that similar problems may be studied over, for example, the \emph{Gaussian integers} $\bbZ[i]$, the \emph{Eisenstein integers} $\bbZ[\exp(2\pi\iu/3)]$ or more generally the \emph{cyclotomic integers} $\bbZ[\exp(2\pi\iu/n)]$. 
The development of a general theory for the \emph{ring of integers} $\cO_K$ of an \emph{algebraic number field} $K \spse \bbQ$ then lead to \emph{algebraic number theory} and \emph{class field theory}. For details on these topics see \cite{Bec14, Chi09, Coh93, Cox89, Mar18, Nar04, Neu99, Was97}. Here we only need to look at the concrete example of the properties of $\bbZ[\phi] \sbse \bbQ(\phi) = \bbQ(\sqrt5)$.

An important tool in the study of the primes in $\bbZ \sbse \bbQ$ is the Riemann zeta function $\zeta(\sigma)$ as already appeared in \eqref{eq:dft_sigma}. The \emph{Dedekind zeta function} $\zeta_K(\sigma)$ is an analogous object for general number fields $K \spse \bbQ$, where $\zeta = \zeta_\bbQ$. We will be interested in $\zeta_{\bbQ(\sqrt5)}$, so we start with collecting some properties of $\bbQ(\sqrt5)$. For general information on Dedekind zeta functions see \cite{Mar18, Nar04, Zag81}.

The field $\bbQ(\sqrt5) \cong \bbQ(x)/(x^2-5)$ is a degree-two extension of the rationals $\bbQ$ and thus an algebraic number field. Its ring of integers is given by $\cO_{\bbQ(\sqrt5)} = \bbZ[\phi]$ and we will give a brief exposition of its properties. Elements in this ring are given by $a+b\phi$ for $a, b \in \bbZ$ and multiplied via
$$
(a+b\phi)(c+d\phi) = (ac+bd) + (ad+bc+bd)\phi.
$$
This ring is a \emph{Euclidean domain} (with Euclidean function $|\nu(x)|$ as below), in particular a \emph{principal ideal domain} and thus also a \emph{unique factorization domain}. The map
$$
\nu: \bbZ[\phi] \ra \bbZ, \nu(a+b\phi) \coloneqq a^2 + ab - b^2
$$
is called the \emph{norm}. It is multiplicative $\nu(xy) = \nu(x) \nu(y)$ and if $(x) \sbse \bbZ[\phi]$ denotes the ideal generated by $x \neq 0$ then for the quotient ring it holds
\begin{align} \label{eq:quot_card}
    \#(\bbZ[\phi]/(x)) = |\nu(x)|.
\end{align}
Via the factorization $\nu(a+b\phi) = (a+b\phi)(a+b-b\phi)$, the element $(a+b)-b\phi$ may be regarded as the conjugate of $a+b\phi$ in $\bbZ[\phi]$. The \emph{units} of $\bbZ[\phi]$ are precisely those elements $u \in \bbZ[\phi]$ for which $\nu(u) \in \{-1, 1\}$ and can be given explicitly by
\begin{align} \label{eq:units}
    \nu^{-1}(1) = \{\pm \phi^n: n \in \bbZ \text{ even}\}, \quad \nu^{-1}(-1) = \{\pm \phi^n: n \in \bbZ \text{ odd}\},
\end{align}
noting that $\phi^n = F_{n-1} + F_{n}\phi$. This makes $\phi$ a \emph{fundamental unit} of $\bbZ[\phi]$. The prime elements of $\bbZ[\phi]$ can be characterized (up to multiplication by units) as follows:
\begin{itemize}
    \item [(i)] $\sqrt5 = -1+2\phi \in \bbZ[\phi]$ is prime,

    \item [(ii)] if $p \in \bbN$ is prime and $p \equiv 2, 3 \mod 5$ then $p \in \bbZ[\phi]$ is prime,

    \item [(iii)] if $p \in \bbN$ is prime and $p \equiv 1, 4 \mod 5$ then there are $a, b \in \bbZ$ with $p=a^2+ab-b^2$ and $a+b\phi,(a+b)-b\phi \in \bbZ[\phi]$ are distinct primes.
\end{itemize}
That $\sqrt5 \in \bbZ[\phi]$ needs to be treated separately is reflected by the fact that $(5) \sbse \bbZ$ is the only prime ideal which ramifies as $(5) = (\sqrt5)^2 \sbse \bbZ[\phi]$.

Points (ii) and (iii) can be seen as an analogue of Fermat's theorem on the sum of two squares for the binary quadratic form $x^2+xy-y^2$ \cite{Bue89, Coh93, Cox89}. This shows that the structure of $\bbZ[\phi]$ is intimately connected to the properties of this form. Indeed, the fact that $\bbZ[\phi]$ is a unique factorization domain (in terms of class field theory the field $\bbQ(\sqrt5)$ has \emph{class number} $h(\bbQ(\sqrt5))=1$) is equivalent to the fact that every binary quadratic form $q(x, y) = ax^2+bxy+cy^2$ with $a, b, c \in \bbZ$ of \emph{discriminant} $b^2-4ac = 5$ is equivalent to $x^2+xy-y^2$ \cite{Coh93}, meaning that
$$
q(\alpha x + \beta y, \gamma x+\beta y) = x^2+xy-y^2
$$
for some $\alpha, \beta, \gamma, \delta \in \bbZ$ with $ \alpha\delta-\beta\gamma=1$.

The \emph{Dedekind zeta function} of $\bbQ(\sqrt5)$ is defined as
$$
\zeta_{\bbQ(\sqrt5)}(\sigma) \coloneqq \sum_{0 \neq I \sbse \bbZ[\phi]} \frac1{\#(\bbZ[\phi] / I)^\sigma}
$$
for $\sigma > 1$, where the sum goes over all non-zero ideals of $\bbZ[\phi]$. The convergence of this sum for $\sigma > 1$ is non-trivial but seen from Proposition \ref{prop:Dedekind_eta} below using $\eta_i \asymp i$. One often studies such functions by first extending them to all of $\bbC \sm \{1\}$ via analytic continuation, but we will not need this here. Using the above characterization of the prime elements of $\bbZ[\phi]$, $\zeta_{\bbQ(\sqrt5)}$ admits the factorization into the \emph{Euler product}
\begin{align} \label{eq:zeta_factorization}
    \begin{split}
        \zeta_{\bbQ(\sqrt5)}(\sigma) & = \prod_{P \sbse \bbZ[\phi]} \frac1{1-\#(\bbZ[\phi] / P)^{-\sigma}} \\
        & = \frac1{1-5^{-\sigma}} \prod_{\substack{p \equiv 1, 4 \mod 5}} \left(\frac1{1-p^{-\sigma}}\right)^2 \prod_{\substack{p \equiv 2, 3 \mod 5}} \frac1{1-p^{-2\sigma}} \\
        & = L(\sigma, \chi) \zeta(\sigma),
    \end{split}
\end{align}
where the first product goes through all prime ideals $P$ of $\bbZ[\phi]$, the products in the second line go through all primes in $\bbN$ and in the last line
$$
L(\sigma, \chi) \coloneqq \sum_{n=1}^\infty \frac{\chi(n)}{n^\sigma}
$$
is the \emph{Dirichlet L-function} for the \emph{character} 
$$
\chi(n) \coloneqq \begin{cases}
    0, & n \equiv 0 \mod 5, \\
    1, & n \equiv 1, 4 \mod 5, \\
    -1, & n \equiv 2, 3 \mod 5,
\end{cases}
$$
that is the \emph{Legendre symbol} $(\frac\cdot5)$ modulo $5$. With $\eta_i$ as in Theorem \ref{thm:wythoff_basics} (ii) the following result connects our previous considerations to $\zeta_{\bbQ(\sqrt5)}(\sigma)$.

\begin{prop} \label{prop:Dedekind_eta}
    For all $\sigma > 1$ it holds
    $$
    \zeta_{\bbQ(\sqrt5)}(\sigma) = \sum_{i=1}^\infty \frac1{\eta_i^\sigma}.
    $$
\end{prop}
\begin{proof}
    We will show that there is a bijection between the rows of the Wythoff array $(W_{i, k})_k$ and the non-zero ideals $I$ of $\bbZ[\phi]$ such that $\#(\bbZ[\phi]/I) = \eta_i$. This then establishes the claim.

    For $i \in \bbN$ consider the ideal $I_i \coloneqq (W_{i, 1} + W_{i, 2}\phi) \sbse \bbZ[\phi]$. Note that $W_{i, k} + W_{i, k+1} \phi = \phi^{k-1} (W_{i, 1} + W_{i, 2}\phi)$ differ only by a unit, so that even $I_i = (W_{i, k} + W_{i, k+1} \phi)$ for all $k \in \bbN$. By \eqref{eq:quot_card} and Theorem \ref{thm:wythoff_basics} (ii) we have $\#(\bbZ[\phi]/I_i) = \eta_i$. It remains to show that the map $i \mapsto I_i$ from $\bbN$ to the set of non-zero ideals of $\bbZ[\phi]$ is a bijection.
    
    To see that it is surjective, note that if $0 \neq I \sbse \bbZ[\phi]$ is a non-zero ideal, since $\bbZ[\phi]$ is a principal ideal domain we have $I = (a+b\phi)$ for some $a, b \in \bbZ, (a, b) \neq (0, 0)$. Multiplying this generator by a suitable unit
    $$
    \pm \phi^n (a+b\phi) = \pm (F_{n-1}a+F_nb) \pm (F_na+F_{n+1}b)\phi,
    $$
    we may assume $a, b > 0$. Consider now the sequence $(U_n)_n$ given by $U_1 = a, U_2 = b$ and $U_n = U_{n-1} + U_{n-2}$. By Theorem \ref{thm:wythoff_basics} (v) there are $i, n, k \in \bbN$ such that $(U_n, U_{n+1}) = (W_{i, k}, W_{i, k+1})$. But then
    $$
    I = (a+b\phi) = (U_1+U_{2}\phi) = (U_n+U_{n+1}\phi) = (W_{i, k}+W_{i, k+1}\phi) = I_i,
    $$
    showing surjectivity.

    For injectivity let $i, i' \in \bbN$ be such that $I_i = I_{i'}$. Then the respective generators $W_{i, 1}+W_{i, 2}\phi$ and $W_{i', 1}+W_{i', 2}\phi$ only differ in a unit $\pm \phi^{n}$, so that the sequences $(W_{i, k})_k$ and $(W_{i', k})_k$ must eventually coincide. However, by Theorem \ref{thm:wythoff_basics} (v), in particular the uniqueness of the index $i$, we then get $i=i'$, showing injectivity. This fully proves the claim.
\end{proof}

This shows that the constant $C$ as in \eqref{C} can be calculated via $\zeta_{\bbQ(\sqrt5)}(\sigma)$. For $\sigma = 2s, s \in \bbN$ an even integer, by \eqref{eq:zeta_factorization} we have (with the Bernoulli numbers $B_{2s}$ and Bernoulli polynomials $B_{2s}(t)$)
$$
\zeta_{\bbQ(\sqrt5)}(2s) = \frac{(B_{2s}(1/5) - B_{2s}(2/5)) B_{2s}}{2 (2s)!^2 \sqrt5} (2\pi)^{4s},
$$
see \cite{Neu99, Was97}. Thus, for the constant $C$ in Theorem \ref{thm:asymp} we generally have
\begin{align} \label{eq:C_exact}
    C = 2 f(0)^2 \frac{5^{\sigma / 2}}{\pi^{2\sigma}} \zeta_{\bbQ(\sqrt5)}(\sigma),
\end{align}
which for $\sigma = 2s, s \in \bbN$ reads
$$
C = f(0)^2 \frac{80^s (B_{2s}(1/5) - B_{2s}(2/5)) B_{2s}}{(2s)!^2 \sqrt5} \sim 2f(0)^2 \left(\frac5{\pi^4}\right)^s.
$$
For even integers and $f(0) = 1$ some values are collected in the following tables:
\begin{center}
	\begin{tabular}{r||c|c|c|c|c|c}
		$\sigma$ & $2$ & $4$ & $6$ & $8$ & $10$ & $12$ \\
		\hline
		$C$ & $\frac{4}{15\sqrt5}$ & $\frac{8}{675\sqrt5}$ & $\frac{1072}{1771875\sqrt5}$ & $\frac{5776}{186046875\sqrt5}$ & $\frac{6604016}{4144194140625\sqrt5}$ & $\frac{25449165152}{311125375107421875\sqrt5}$
	\end{tabular}
\end{center}
\begin{center}
	\begin{tabular}{r||c|c|c}
		$\sigma$ & $14$ & $16$ & $18$ \\
		\hline
		$C$ & $\frac{36389877952}{8667064020849609375\sqrt5}$ & $\frac{1750445666277664}{8122122370538690185546875\sqrt5}$ & $\frac{9141810707034331408}{826385340590459032928466796875\sqrt5}$
	\end{tabular}
\end{center}
For $\sigma = 2s = 2$ we recover the value from \cite{Bor17, HKP21}. While the expression \eqref{D} for the constant $D$ is significantly more complicated, the next section will show that for special (yet still relevant) cases a simple closed-form expression may be achievable.

\section{Closed-form expressions}

It is natural to ask if one can determine even more lower-order terms beyond the constant one. With a more detailed analysis and possibly stronger assumptions on the functions $f$ this seems doable, but tedious without a more systematic approach. In this last part we want to discuss a different and to the authors surprising aspect, that one can even determine simple closed-form expressions for certain sums of this type. With the \emph{Lucas numbers} $(L_n)_n$ given by $L_0=2, L_1=1$ and the recurrence $L_n=L_{n-1}+L_{n-2}$ for $n \geq 2$, we exhibit the following.

\begin{thm} \label{thm:closed_form}
    For all $n \geq 2$ it holds
    $$
    \sum_{m=1}^{F_n-1} \frac1{\sin(\pi m/F_n)^2} \frac1{\sin(\pi F_{n-1}m/F_n)^2} = \frac{4n}{75} F_{2n} - \frac{17}{1125}L_{2n} - (-1)^n \frac{116}{1125} - \frac19.
    $$
\end{thm}

Although this sum in particular already showed up for example in \cite{BTY12, Bor17, HKP21}, its simple closed-form expression has not been noted before. For the proof we remind of the explicit expressions for the Bernoulli polynomials
$$
B_1(t) = t-\frac12, \quad B_2(t) = t^2-t+\frac16, \quad B_3(t) = t^3-\frac32t^2+\frac12t.
$$

\begin{proof} [Proof of Theorem \ref{thm:closed_form}]
    By \eqref{eq:Bernoulli_dft} with $2s=2$ and Proposition \ref{prop:dft_sum} the desired formula is equivalent to
    \begin{align} \label{eq:S22}
        \begin{split}
            & \frs^{}_{2, 2}(1, F_{n-1}; F_n) \\
            = & \sum_{k = 0}^{F_n-1} B_2\left(\frac k{F_n}\right) B_2\left(\left\{\frac{F_{n-1} k}{F_n}\right\}\right) = \frac{n F_{2n}}{75 F_n^3} - \frac{17L_{2n}}{4500F_n^3}-(-1)^n\frac{29}{1125 F_n^3}.
        \end{split}
    \end{align}
    In general, for $\ell, m \in \bbN$ and $a, b, c \in \bbN$ with $\gcd(a, c) = \gcd(b, c) = 1$ define the \emph{generalized Dedekind sum}
    $$
    \frs^{}_{\ell, m}(a, b; c) \coloneqq \sum_{k = 0}^{c-1} B_\ell\left(\left\{\frac{a k}{c}\right\}\right) B_m\left(\left\{\frac{b k}{c}\right\}\right).
    $$
    We will show the above expression \eqref{eq:S22} for $\frs^{}_{2,2}(1, F_{n-1}; F_n)$ as well as
    \begin{align} \label{eq:S13}
	\begin{split}
	    & \frs^{}_{1, 3}(1, F_{n-1}; F_n) = (-1)^n \frs^{}_{3, 1}(1, F_{n-1}; F_n) \\
	= ~ & \frac{L_{3n}}{1500F_n^3} + (-1)^n\frac{n}{50F_n^2} - (-1)^n \frac{13L_n}{750F_n^3}
	\end{split}
    \end{align}
    via reciprocity laws for generalized Dedekind sums \cite{Apo52, Bec03, BC11, BY02, BZ04, Die57, Die84, Die96, HWZ95, Mac08, Mik57, RG72, Tak79}.
    
    We start with \eqref{eq:S13} which we will prove via induction (the relation $\frs^{}_{1, 3}(1, F_{n-1}; F_n) = (-1)^n \frs^{}_{3, 1}(1, F_{n-1}; F_n)$ being a simple consequence of the congruence
    $$
    F_{n-1}^2 \equiv (-1)^n \mod F_n
    $$
    together with the fact that $B_3(t) = -B_3(1-t)$ is odd). Both sides of \eqref{eq:S13} evaluate to $0$ if $n=2$ and we may continue with the induction step. From \cite{Apo52} we get
    $$
    4 \left( bc^3\frs^{}_{1, 3}(1, b; c) + b^3c \, \frs^{}_{1,3}(1,c;b) \right) = -\frac1{10}-\frac1{30}\left(b^4 -5b^2c^2+c^4\right)
    $$
    if $\gcd(b, c) = 1$. For $b = F_{n-1}$ and $c = F_n$, using $\frs^{}_{1, 3}(1, F_{n}; F_{n-1}) = \frs^{}_{1, 3}(1, F_{n-2}; F_{n-1})$ by $F_n-F_{n-1} = F_{n-2}$, we consequently have
    \begin{align} \label{eq:ind_step_13}
        \begin{split}
            4F_{n-1} F_n^3 \, \frs^{}_{1, 3}(1, F_{n-1}; F_n) = & -\frac1{10} - \frac1{30} \left(F_{n-1}^4 - 5 F_{n-1}^2 F_n^2 + F_n^4\right) \\
        & - 4F_{n-1}^3 F_n \, \frs^{}_{1, 3}(1, F_{n-2}; F_{n-1}).
        \end{split}
    \end{align}
    We make use of the relations (verified via \eqref{eq:binet} and the analogous $L_n = \phi^n + (-\phi)^{-n}$)
    \begin{align*}
        \frac{L_{3n}}{375} F_{n-1} & = \frac{17}{750} - \frac1{30} \left(F_{n-1}^4 - 5 F_{n-1}^2 F_n^2 + F_n^4\right) - \frac{L_{3n-3} F_n}{375} - (-1)^n \frac{2 L_{2n-1}}{125}, \\
        \frac{26L_n}{375}F_{n-1} & = (-1)^n \frac{46}{375} + \frac{2}{25} F_n F_{n-1} + \frac{26}{375} F_n L_{n-1} - \frac{2L_{2n-1}}{125},
    \end{align*}
    and the induction hypothesis to continue \eqref{eq:ind_step_13} and obtain
    \begin{align*}
        & 4F_{n-1} F_n^3 \, \frs^{}_{1, 3}(1, F_{n-1}; F_n)
        \\ =~& -\frac1{10} - \frac1{30} \left(F_{n-1}^4 - 5 F_{n-1}^2 F_n^2 + F_n^4\right) - 4F_{n-1}^3 F_n \, \frs^{}_{1, 3}(1, F_{n-2}; F_{n-1}) \\
        = ~& -\frac1{10} - \frac1{30} \left(F_{n-1}^4 - 5 F_{n-1}^2 F_n^2 + F_n^4\right) \\
        & - 4 F_n \left(\frac{L_{3n-3}}{1500} + (-1)^{n-1}\frac{(n-1)F_{n-1}}{50} - (-1)^{n-1} \frac{13L_{n-1}}{750}\right) \\
        = ~& \frac{17}{750} - \frac1{30} \left(F_{n-1}^4 - 5 F_{n-1}^2 F_n^2 + F_n^4\right) - 4 F_n \cdot \frac{L_{3n-3}}{1500} - (-1)^n \frac{2 L_{2n-1}}{125} \\
        & - 4 F_n \cdot (-1)^{n-1} \frac{n F_{n-1}}{50} \\
        & -\frac{46}{375} -4F_n \left(-(-1)^{n-1} \frac{F_{n-1}}{50} - (-1)^{n-1} \frac{13 L_{n-1}}{750}\right) + (-1)^n \frac{2 L_{2n-1}}{125} \\
        = ~& \frac{L_{3n}}{375} F_{n-1} + (-1)^n\frac{2n F_n}{25} F_{n-1} - (-1)^n \frac{26L_n }{375} F_{n-1}.
    \end{align*}
    Division by $4 F_{n-1} F_n^3$ gives \eqref{eq:S13}.

    We will deduce \eqref{eq:S22} from the now established \eqref{eq:S13}. 
    From \cite{HWZ95} we use the relation
    \begin{align*} 
        \begin{split}
            & \frac1{4b} \, \frs^{}_{2,2}(1, b; c) \\
        = ~ & \frac16\left(\frs^{}_{3, 1}(1, b; c) + \frs^{}_{3, 1}(1, c; b)\right) + \frac1{6b^2} \frs^{}_{1,3}(1, b; c) + \frac1{720 b^3 c^3} + \frac b{720 c^3} + \frac c{240 b^3}
        \end{split}
    \end{align*}
    for relatively prime integers $\gcd(b, c) = 1$. For $b=F_{n-1}$ and $c=F_n$ this gives
    \begin{align*}
        & \frac1{4F_{n-1}} \frs^{}_{2,2}(1, F_{n-1}; F_n) \\
        = ~ & \frac16\left(\frs^{}_{3, 1}(1, F_{n-1}; F_n) + \frs^{}_{3, 1}(1, F_{n-2}; F_{n-1})\right) + \frac1{6F_{n-1}^2} \frs^{}_{1,3}(1, F_{n-1}; F_n) \\
        & + \frac1{720 F_{n-1}^3 F_n^3} + \frac {F_{n-1}}{720 F_n^3} + \frac {F_n}{240 F_{n-1}^3}.
    \end{align*}
    All that remains is to insert \eqref{eq:S13} into this expression and to simplify accordingly as to obtain \eqref{eq:S22} (like above, using \eqref{eq:binet} and the analogous relation for $L_n$). These calculations are elementary but lengthy, so we skip them here.
\end{proof}

In terms of $e_{2, p}(\Phi_n)$ we obtain
$$
e_{2, p}(\Phi_n)^2 = \frac p{6F_n^2} + \frac{p^2}{300F_n^4} \left(nF_{2n} - \frac{17}{60} L_{2n} - (-1)^n \frac{29}{15}\right).
$$
Even more, via computer calculations we observe
\begin{align*}
    & \sum_{m=1}^{F_n-1} \frac{2+4\cos(\pi m/F_n)^2}{\sin(\pi m/F_n)^4} \frac{2+4\cos(\pi F_{n-1}m/F_n)^2}{\sin(\pi F_{n-1}m/F_n)^4} \\
    = ~& \frac{32n}{1875} F_{4n} - \frac{196}{28125} L_{4n} + (-1)^n \frac{256n}{1875} F_{2n} - (-1)^n \frac{3776}{28125} L_{2n} - \frac{7556}{28125}
\end{align*}
and
\begin{align*}
    & \sum_{m=1}^{F_n-1} \frac{16+88\cos(\pi m/F_n)^2+16\cos(\pi m/F_n)^4}{\sin(\pi m/F_n)^6} \\
    & ~~~~~~ \times \frac{16+88\cos(\pi F_{n-1}m/F_n)^2+16\cos(\pi F_{n-1}m/F_n)^4}{\sin(\pi F_{n-1}m/F_n)^6} \\
    = ~& \frac{68608n}{984375} F_{6n} + \frac{59606528}{20155078125} L_{6n} + (-1)^n \frac{548864n}{328125} F_{4n} - (-1)^n \frac{2876091392}{2239453125} L_{4n} \\
	& + \frac{68608n}{13125} F_{2n} - \frac{128925952}{17915625} L_{2n} - (-1)^n \left( \frac{1909649408}{161240625} + (-1)^n \frac{256}{3969} \right),
\end{align*}
by Proposition \ref{prop:dft_sum} and \eqref{eq:Bernoulli_dft} corresponding to $\frs^{}_{4,4}(1, F_{n-1}; F_n)$ and $\frs^{}_{6, 6}(1, F_{n-1}; F_n)$ respectively. Similar inductive proofs as above seem possible for these sums too, but the complexity of the according calculations rises quickly as the degrees of the corresponding Dedekind sums increases. Furthermore, we observe
\begin{align*}
	& \sum_{m=1}^{F_n-1} \frac1{\sin(\pi m/{F_n})^4} \frac1{\sin(\pi F_{n-1} m/{F_n})^4} \\
	= ~ & \frac{8n}{16875} F_{4n} + \frac{2357}{1771875} L_{4n} + \left(\frac{16}{675} + (-1)^n \frac{64}{16875}\right) n F_{2n} \\
	& - \left(\frac{676}{70875} + (-1)^n \frac{7408}{1771875} \right) L_{2n} - \left(\frac{147023}{1771875} + (-1)^n \frac{1616}{23625} \right)
\end{align*}
and
\begin{align*}
	& \sum_{m=1}^{F_n-1} \frac{\cos(\pi m/{F_n})^2}{\sin(\pi m/{F_n})^4} \frac{{\cos(\pi F_{n-1} m/{F_n})^2}}{\sin(\pi F_{n-1} m/{F_n})^4} \\
	= ~ & \frac{8n}{16875} F_{4n} - \frac{1693}{1771875} L_{4n} + \left(\frac{4}{675} + (-1)^n \frac{64}{16875}\right) n F_{2n} \\
	& - \left(\frac{19}{70875} + (-1)^n \frac{6208}{1771875} \right) L_{2n} - \left(\frac{11948}{1771875} + (-1)^n \frac{4}{23625} \right).
\end{align*}
While it seems possible to give a general closed-form expression for sums of this type, we do not have such a universal approach at this point. Similar closed-form expressions for the usual Dedekind sums (that is involving Bernoulli polynomials of degree $1$) have been studied in \cite{Cet23, DM10, DEKT22, LZ13, San23, ZY00, ZW04}. As for Dedekind sums of higher degree, we are only aware of \cite{Tak79}, where a formula for $\frs_{1,m}(1,b;c)$ in terms of the continued fraction expansion of $b/c$ is given. This, however, is still less explicit than the expression in \eqref{eq:S13} above.

\section*{Acknowledgment}

The authors were partially funded by the ESF Plus Young Researchers Group ReSIDA-H2 (SAB, 10064931). Nicolas Nagel acknowledges the support of the Austrian Science Fund (FWF) Project  P 34808/Grant DOI: 10.55776/P34808. For open access purposes, the authors have applied a CC BY public copyright license to any author accepted manuscript version arising from this submission.

We would like to thank Bence Borda and Dmitriy Bilyk for helpful comments in the early stages of this manuscript. We would also like to thank Friedrich Pillichshammer and Robert Tichy for useful suggestions.

\section*{Disclosure statement}

The authors report there are no competing interests to declare.

\bibliographystyle{plain}
\bibliography{refs}

\appendix

\section{Auxiliary results}

\subsection{Floor function}

Here we collect some results concerning the floor function that are used throughout Section \ref{sec:wythoff}. Similar relations have been proven and used for example in \cite{Con59, Kim95, Kim20, Mor80}. Recall that for $x \in \bbR$ and $n \in \bbZ$ we have $\lfloor x \rfloor = n$ if and only if $x-n \in [0, 1)$.

\begin{lem} \label{lem:floor}
For all $N \in \bbN$ the following equalities and inequalities hold:
    \begin{itemize}
        \item [(i)] $\phi \lfloor \phi N\rfloor < \lfloor \phi^2 N\rfloor$,

        \item [(ii)] $\lfloor \phi^2 N \rfloor = \lfloor \phi \lfloor \phi N\rfloor + \phi^{-1}\rfloor$,

        \item [(iii)] $N + \lfloor \phi N + \phi^{-1}\rfloor = \lfloor \phi \lfloor \phi N + \phi^{-1}\rfloor + \phi^{-1}\rfloor$,

        \item [(iv)] $\frac12 \left\lfloor \phi (2N-1) + \phi^{-1} \right\rfloor < \left\lfloor \phi N + \phi^{-1} \right\rfloor < \frac12 \left\lfloor \phi (2N+1) + \phi^{-1} \right\rfloor$.
    \end{itemize}
\end{lem}
\begin{proof}
    We will use the relation $\phi^2=\phi+1$ and variations thereof as well as the irrationality of $\phi$ multiple times throughout the proof.

    \textit{(i):} Starting with $\phi^{-1} \lfloor \phi N\rfloor < N$ we have
    $$
    \phi \lfloor \phi N\rfloor = \phi^{-1} \lfloor \phi N\rfloor + \lfloor \phi N\rfloor < N + \lfloor \phi N\rfloor = \lfloor \phi^2 N\rfloor.
    $$

    \textit{(ii):} From $-1 < \lfloor \phi N\rfloor - \phi N < 0$ we get $\phi^{-1}\lfloor \phi N \rfloor + \phi^{-1} - N \in (0, \phi^{-1})$. This gives $\lfloor \phi^{-1}\lfloor \phi N \rfloor + \phi^{-1} \rfloor = N$ and thus
    $$
    \left\lfloor \phi^2 N \right\rfloor = \lfloor \phi N \rfloor + N = \lfloor \phi N \rfloor + \left\lfloor \phi^{-1} \lfloor \phi N\rfloor + \phi^{-1} \right\rfloor = \left\lfloor \phi \lfloor \phi N \rfloor + \phi^{-1} \right\rfloor.
    $$

    \textit{(iii):} Since $\phi N + \phi^{-1} = \frac{N-1}{2} + \frac{N+1}{2}\sqrt5 \notin \bbN$ we have
	$$
	\phi N - \phi^{-2} < \left\lfloor \phi N + \phi^{-1}\right\rfloor < \phi N + \phi^{-1},
	$$
	so that
	\begin{align*}
		\phi \left\lfloor \phi N + \phi^{-1}\right\rfloor + \phi^{-1} - \left(\left\lfloor \phi N + \phi^{-1}\right\rfloor + N\right)
		= \phi^{-1} \left( \left\lfloor \phi N + \phi^{-1}\right\rfloor + 1\right) - N
	\end{align*}
	lies in the interval $(\phi^{-2}, 1)$. This gives the claim.

    \textit{(iv):} For the left inequality note the following sequence of implications:
    \begin{align*}
         & -\phi+\phi^{-1} < 2\phi^{-1} - 2 \\
         \Rightarrow \quad & \phi(2N-1)+\phi^{-1} < 2\phi N+2\phi^{-1}-2 \\
         \Rightarrow \quad & \lfloor \phi(2N-1)+\phi^{-1}\rfloor < 2\phi N +2\phi^{-1}-2 \\
         \Rightarrow \quad & \frac12\lfloor\phi(2N-1)+\phi^{-1}\rfloor+1<\phi N+\phi^{-1} \\
         \Rightarrow \quad & \frac12\lfloor\phi(2N-1)+\phi^{-1}\rfloor < \lfloor \phi N+\phi^{-1} \rfloor.
    \end{align*}
    For the right inequality we similarly have
    \begin{align*}
         & 2\phi^{-1} + 1 = \phi+\phi^{-1} \\
         \Rightarrow \quad & 2\phi N+2\phi^{-1} + 1 = \phi (2N+1)+\phi^{-1} \\
         \Rightarrow \quad & 2\phi N + 2\phi^{-1} < \lfloor \phi (2N+1)+\phi^{-1} \rfloor \\
         \Rightarrow \quad & \phi N + \phi^{-1} < \frac12 \lfloor \phi (2N+1)+\phi^{-1} \rfloor \\
         \Rightarrow \quad & \lfloor\phi N + \phi^{-1}\rfloor < \frac12 \lfloor \phi (2N+1)+\phi^{-1} \rfloor.
    \end{align*}
\end{proof}

The following results are related to how well the golden ratio $\phi$ can be approximated by rational numbers.

\begin{lem} \label{lem:floor_ineq}
    For all $i \in \bbN$ it holds
    $$
    1+\frac1{(\phi+2) i^2} < \frac{\phi i}{\lfloor \phi i\rfloor} \quad \text{and} \quad \frac{\phi i}{\lfloor \phi i\rfloor+1} \leq 1 - \frac1{2\phi^2 i^2}.
    $$
\end{lem}
\begin{proof}
    Let $i \in \bbN$. The first inequality is easily seen to be implied by the stronger statement
    \begin{align} \label{ineq:stronger_0}
        \{\phi i\} > \frac1{\sqrt5 i}.
    \end{align}
    To show this inequality note that, using $\lfloor \phi i\rfloor = \phi i - \{\phi i\}$,
    $$
    \sqrt5i\{\phi i\} - (i^2+i\lfloor\phi i\rfloor - \lfloor \phi i\rfloor^2) = \{\phi i\}^2 > 0,
    $$
    so $\sqrt5i\{\phi i\} > i^2+i\lfloor\phi i\rfloor - \lfloor \phi i\rfloor^2$. The expression on the right hand side is an integer which by
    $$
    i^2+i\lfloor\phi i\rfloor - \lfloor \phi i\rfloor^2 = (i-\phi^{-1}\lfloor\phi i\rfloor) (i+\phi\lfloor\phi i\rfloor) > 0
    $$
    (since $\phi$ is irrational) is positive, that is at least $1$. This gives \eqref{ineq:stronger_0} and thus the first inequality of the statement.

    For the second inequality, which is implied by (using $1+\lfloor \phi i\rfloor \leq 2 i$)
    \begin{align} \label{ineq:stronger_1}
        1-\{\phi i\} \geq \frac1{\phi^2 i},
    \end{align}
    we proceed in a similar manner to above. Noting that $\lceil \phi i\rceil = \lfloor\phi i\rfloor + 1 = \phi i + 1 - \{\phi i\}$ by irrationality of $\phi$, this follows from
    $$
    \phi^2 i (1 - \{\phi i\}) - (\lceil \phi i\rceil^2 - i \lceil \phi i\rceil - i^2) = (1-\{\phi i\}) (\phi^{-2} i - (1-\{\phi i\})) \geq 0,
    $$
    so that analogously to above we get
    $$
    \phi^2 i (1 - \{\phi i\}) \geq \lceil \phi i\rceil^2 - i \lceil \phi i\rceil - i^2 \geq 1,
    $$
    giving \eqref{ineq:stronger_1}. This establishes both inequalities in the statement.
\end{proof}

Note that the auxiliary sequences $i^2+i\lfloor\phi i\rfloor - \lfloor \phi i\rfloor^2$ and $\lceil \phi i\rceil^2 - i \lceil \phi i\rceil - i^2$ together already showed up in \cite{Kim075}.

\subsection{Basic inequalities}

Here we collect some inequalities from elementary calculus that are used throughout Section \ref{sec:asymp}.

\begin{lem} \label{lem:calc_ineqs}
    The following inequalities hold:
    \begin{itemize}
        \item [(i)] For $0 < x \leq \frac23$ it holds
        $$
        \frac1{\sin(\pi x)} \leq \frac{1}{x}.
        $$

        \item [(ii)] For $\sigma > 1$ and $0 < x, y \leq \frac23$ it holds
        $$
        \left|\frac1{\sin(\pi x)^\sigma} - \frac1{\sin(\pi y)^\sigma}\right| \leq \sigma \pi^\sigma \left(\frac1x + \frac1y\right)^{\sigma-1} \left|\frac1x - \frac1y\right|.
        $$

        \item [(iii)] For $\sigma > 1$ and $0 < x \leq \frac23$ it holds
        $$
        \left(\frac{\pi x}{\sin(\pi x)}\right)^{\sigma}-1 \leq  4^{\sigma+1} x^2.
        $$

        \item [(iv)] For $\sigma > 1$ and $|x| \leq \frac23$ it holds
        $$
        \left|\left(\frac1{1-x}\right)^\sigma - 1\right| \leq 4^\sigma |x|.
        $$
    \end{itemize}
\end{lem}
\begin{proof}
    \textit{(i):} The function $\sin(\pi x)$ is concave on the interval $x \in [0, \frac23]$ (by non-positivity of its second derivative there), so that
    $$
    \sin(\pi x) \geq \frac{3}{2} \sin\left(\frac{2\pi}3\right) x = \frac{3\sqrt3}{4} x \geq x
    $$
    for all $0 \leq x \leq \frac23$. The statement now follows by taking the inverses.

    \textit{(ii):} 
    If $x=y$, both sides are 0 and the inequality is obviously true. Therefore, without loss of generality let $x<y$ and consider the ratio
    $$R(x,y) \coloneqq \frac{\left|\frac1{\sin(\pi x)^\sigma} - \frac1{\sin(\pi y)^\sigma}\right|}{\left(\frac1x + \frac1y\right)^{\sigma-1} \left|\frac1x - \frac1y\right|} = \frac{x^{\sigma}y^\sigma |\sin(\pi y)^\sigma-\sin(\pi x)^\sigma|}{\sin(\pi x)^\sigma\sin(\pi y)^\sigma (x+y)^{\sigma-1}|y-x|}.$$
    By (i), $x^\sigma\sin(\pi x)^{-\sigma}\leq 1$, respectively for $y$. Furthermore, the mean value theorem guarantees the existence of a $\xi\in[x,y]$ with
    $$
    \sin(\pi y)^{\sigma}-\sin(\pi x)^{\sigma} = \sigma\pi\sin(\pi \xi)^{\sigma-1}\cos(\pi\xi)(y-x),
    $$
    so that $$\left| \frac{\sin(\pi y)^\sigma-\sin(\pi x)^\sigma}{y-x}\right|=|\sigma\pi\sin(\pi \xi)^{\sigma-1}\cos(\pi\xi)|\leq\sigma\pi(\pi\xi)^{\sigma-1}\leq\sigma\pi^\sigma (x+y)^{\sigma-1}$$ since $\sin(\pi x)\leq \pi x$ for $x\geq0$ (by concavity, considering the tangent at $x=0$), $\sigma-1>0$ and $\xi\leq y < x+y$. 
    All in all, it follows that $R(x,y)\leq \sigma\pi^\sigma$, hence the assertion.

    \textit{(iii):} Consider the function $(0, \frac{2}{3}] \rightarrow\bbR, x\mapsto x^{-2}((1-\frac{\pi^{2}x^{2}}{6})^{-\sigma}-1)$. We have the expansion
    $$
    \frac1{x^2}\left(\frac1{\left(1-\frac{\pi^2}6 x^2\right)^\sigma}-1\right) = \sum_{n=0}^\infty {n+\sigma \choose n+1} \left(\frac{\pi^2}{6}\right)^{n+1} x^{2n}
    $$
    for $0 < x \leq \frac23$ and by positivity of the coefficients it becomes apparent that this function is increasing on $(0, \frac23]$. One then has
    $$
    \frac1{x^2}\left(\frac1{\left(1-\frac{\pi^2}6 x^2\right)^\sigma}-1\right) \leq \frac1{(2/3)^2}\left(\frac1{\left(1-\frac{\pi^2}6 (2/3)^2\right)^\sigma}-1\right) \leq  4^{\sigma+1}
$$
for $0 < x \leq \frac23$.
    
    Now it is easily seen that $\sin(\pi x)\geq\pi x -\frac{\pi^{3}x^{3}}{6}$ for all $x\geq 0$ (equality for $x=0$ and consider the derivatives up to order 3). Consequently, for all $x\in(0,\frac{2}{3}]$ we obtain the desired
    \begin{align*}
        \left(\frac{\pi x}{\sin(\pi x)}\right)^{\sigma}-1 &\leq \left(\frac{\pi x}{\pi x -\frac{\pi^{3}x^{3}}{6}}\right)^{\sigma}-1 = \left(1-\frac{\pi^{2}x^{2}}{6}\right)^{-\sigma}  -1 \leq 4^{\sigma+1} x^{2}.
    \end{align*}

    \textit{(iv):} The function $(-\infty, 1) \ra \bbR, x \mapsto (1-x)^{-\sigma}-1$ is convex in this range (by a simple check of its second derivative), in particular for all $0 \leq x \leq \frac23$ it holds
    $$
    \left(\frac1{1-x}\right)^\sigma - 1 \leq \frac32 (3^\sigma-1)x \leq 4^\sigma x.
    $$
    On the other hand, $1-(1-x)^{-\sigma}$ is concave for $x \in (-\infty, 1)$ (as the negative of the convex function above). By considering the tangent at $x=0$ we thus get for $x\leq0$
    $$
    1-\left(\frac1{1-x}\right)^\sigma \leq -\sigma x \leq -4^\sigma x.
    $$
    This fully establishes the desired inequality.
\end{proof}

\end{document}